\documentclass[aop,seceqn,citesort,MSNbibl,dvips]{arximspdf}

\doi{10.1214/09-AOP503}
\volume{38}
\issue{3}
\pubyear{2010}
\firstpage{1221}
\lastpage{1262}

\makeatletter

\newtheorem{theorem}{Theorem}[section]
\newtheorem{proposition}[theorem]{Proposition}
\newtheorem{corollary}[theorem]{Corollary}
\newtheorem{lemma}[theorem]{Lemma}

\newproclaim{remark}{Remark}[section]
\newproclaim{conjecture}{Conjecture}[section]

\makeatother

\begin{document}
\begin{frontmatter}

\title{Askey--Wilson polynomials, quadratic harnesses and martingales}
\runtitle{Askey--Wilson polynomial martingales}

\begin{aug}
\author[A]{\fnms{W{\l}odek} \snm{Bryc}\corref{}\thanksref{t1}\ead[label=e1]{Wlodzimierz.Bryc@UC.edu}\ead[label=u1,url]{http://math.uc.edu/\texttildelow brycw}} and
\author[B]{\fnms{Jacek} \snm{Weso{\l}owski}\ead[label=e2]{wesolo@mini.pw.edu.pl}\ead[label=u2,url]{http://www.mini.pw.edu.pl/tiki-index.php?page=prac\_wesolowski\_jacek}}
\runauthor{W. Bryc and J. Weso{\l}owski}
\affiliation{University of Cincinnati and Warsaw University of Technology}
\address[A]{Department of Mathematics\\
University of Cincinnati\\
PO Box 210025\\
Cincinnati, Ohio 45221--0025\\
USA\\
\printead{e1}\\
\printead{u1}}
\address[B]{Faculty of Mathematics\\
\quad and Information Science\\
Warsaw University of Technology\\
Pl. Politechniki 1\\
00-661 Warszawa\\
Poland\\
\printead{e2}}
\end{aug}

\thankstext{t1}{Supported in part by NSF
Grant DMS-09-04720 and by Taft Research Seminar 2008-09.}

% HISTORY:
\received{\smonth{12} \syear{2008}}
\revised{\smonth{7} \syear{2009}}

% ABSTRACT
%
\begin{abstract}
We use orthogonality measures of Askey--Wilson polynomials to construct
Markov processes with linear regressions and quadratic conditional
variances. Askey--Wilson polynomials are orthogonal martingale
polynomials for these processes.
\end{abstract}

% KEYWORDS
%

\end{frontmatter}

%s1 ###
\section{Introduction}
Orthogonal martingale polynomials for stochastic processes have been
studied by a number of authors (see
\cite
{BakryMazet03,Feinsilver86,Lytvynov03,NualartSchouten00,Schoutens00,SchoutensTeugels98,SoleUtzet08,SoleUtzet08b}).
Orthogonal martingale polynomials play also a prominent role in
noncommutative probability
\cite{Anshelevich01,Anshelevich04IMRN} and can serve
as a connection to the so called ``classical versions'' of
noncommutative processes. On the other hand, classical versions may
exist without polynomial martingale structure (see \cite{Biane98}).
In \cite{BrycMatysiakWesolowski04} we identify intrinsic
properties of the first two conditional moments of a
stochastic process that guarantee the
process has orthogonal martingale
polynomials. These properties, linear conditional expectations and
quadratic conditional variances, which we call the quadratic harness
properties, have already lead to a number of new examples of Markov
processes
\cite{BrycMatysiakWesolowski04b,BrycWesolowski03,BrycWesolowski04}
with orthogonal martingale polynomials. Random fields with harness
properties were introduced by Hammersley \cite{Hammersley} and their
properties were studied (see, e.g., \cite{MansuyYor04,Williams73}).

In this paper we use measures of orthogonality of Askey--Wilson
polynomials to construct a large class of Markov processes with
quadratic harness properties that includes most of the previous examples,
either as special cases or as ``boundary cases.''
The main step is the construction of an auxiliary Markov process which
has Askey--Wilson polynomials \cite{AskeyWilson85} as
orthogonal martingale polynomials.
The question of probabilistic interpretation
of Askey--Wilson polynomials was raised in \cite{Diaconis08r}, page 197.

The paper is organized as follows. In the remainder of this section we
recall background material on the
quadratic harness property and Askey--Wilson polynomials; we also state
our two main results. In Section \ref{S:elementary} we give an elementary
construction that does not cover the entire range of parameters, but
it is explicit and does not rely on orthogonal polynomials. The
general construction appears in Section \ref{S:general}; this proof
follows the method from \cite{BrycMatysiakWesolowski04b} and
relies on the martingale property of Askey--Wilson polynomials which
extends a projection formula from \cite{NassrallahRahman85} to a
larger range of parameters.
Section \ref{S:discrete} contains another elementary but
computationally more
cumbersome construction of a purely discrete quadratic harness which is
not covered by Theorem \ref{Thm-QH-q}. Section \ref{S:WoE}
illustrates how some of our previous constructions, and some new cases,
follow from Theorem \ref{Thm-QH-q} essentially by a calculation.
In the \hyperref[app]{Appendix}, we discuss two results on orthogonal
polynomials in
the form we need in this paper: a version of Favard's theorem that does
not depend on the support of the orthogonality measure and a version of
connection coefficients formula for Askey--Wilson polynomials, \cite
{AskeyWilson85}.

%s1.1 ###
\subsection{Quadratic harnesses}\label{S:QH}
In \cite{BrycMatysiakWesolowski04} the authors consider
square-integrable stochastic processes on $(0,\infty)$
such that
for all $t,s> 0$,
%
%e1.1 ###
%
\begin{equation}\label{EQ:cov}
{\mathbb E}(X_t)=0,\qquad {\mathbb E}(X_tX_s)=\min\{t,s\},
\end{equation}
$E({X_t}|{\mathcal{F}_{ s, u}})$ is a linear function of $X_s,X_u$ and
$\operatorname{Var}[X_t|\mathcal{F}_{s,u }]$ is a quadratic function
of $X_s,X_u$.
Here, $\mathcal{F}_{ s, u}$ is the two-sided
$\sigma$-field generated by $\{X_r\dvtx r\in(0,s]\cup[u,\infty)\}$. We
will also use the one-sided $\sigma$-fields $\mathcal{F}_{ t }$
generated by $\{X_r\dvtx r\leq t\}$.

Then for all $s<t<u$,
%
%e1.2 ###
%
\begin{equation}
\label{EQ:LR} {\mathbb E}({X_t}|{\mathcal{F}_{ s, u}})=\frac{u-t}{u-s}
X_s+\frac{t-s}{u-s} X_u
\end{equation}
and under certain technical assumptions, Bryc, Matysiak and Wesolowski
\cite{BrycMatysiakWesolowski04}, Theorem 2.2,
assert that there exist numerical constants
$\eta,\theta,\sigma,\tau,\gamma$ such that
for all $s<t<u$,
%
%e1.3 ###
%
\begin{eqnarray}\label{EQ:q-Var}\quad
&&\operatorname{Var}[X_t|\mathcal{F}_{s,u }]\nonumber\\
&&\qquad= \frac{(u-t)(t-s)}{u(1+\sigma s)+\tau-\gamma s} \biggl( 1+ \sigma
\frac{(uX_s-sX_u)^2}{(u-s)^2}+\eta\frac{uX_s-sX_u}{u-s}
\nonumber\\[-8pt]\\[-8pt]
&&\hspace*{130.5pt}{} + \tau\frac{(X_u-X_s)^2}{(u-s)^2}+\theta\frac{X_u-X_s}{u-s}
\nonumber\\
&&\hspace*{138.7pt}{} -(1-\gamma)\frac{(X_u-X_s)(uX_s-sX_u)}{(u-s)^2} \biggr).\nonumber
\end{eqnarray}
%
% In this paper it is convenient to extend
% defined on a time domain $T$ which
%may be a proper subset of $[0,\infty)$, and we
%relax condition \eqref{EQ: cov} to allow more general covariances.
%(The latter is not a significant generalization as according to
%under assumption \eqref{EQ: LR} can only be generalized to
%%\begin{definition}
%Accordingly, w
We will say that a square-integrable stochastic process $(X_t)_{t\in
T}$ is a quadratic harness on $T$ with parameters $(\eta,\theta
,\sigma,\tau,\gamma)$
if it satisfies (\ref{EQ:LR}) and (\ref{EQ:q-Var}) on $T$ which may
be a proper subset of $(0,\infty)$. In previous papers (see, e.g.,
\cite{BrycMatysiakWesolowski04}) only $T=(0,\infty)$ was considered.

Under the conditions listed in \cite{BrycMatysiakWesolowski04},
Theorems 2.4 and 4.1, quadratic harnesses on
$(0,\infty)$ have orthogonal martingale polynomials. Although
several explicit three-step recurrences have been worked out in \cite
{BrycMatysiakWesolowski04},
Section 4, and even though, for some of the
recurrences, corresponding
quadratic harnesses were constructed in a series of papers
\cite{BrycMatysiakWesolowski04b,BrycWesolowski03,BrycWesolowski04},
the general orthogonal martingale polynomials have not been
identified, and the question of existence of corresponding quadratic
harnesses was left open.

It has been noted that the family of all quadratic harnesses on
$(0,\infty)$ that satisfy condition (\ref{EQ:cov}) is invariant
under the action of
translations and reflections of ${\mathbb R}$: translation by $a\in
{\mathbb R}$
acts as $(X_t)\mapsto e^{-a}X_{e^{2a}t}$ and the reflection at $0$ acts
as $(X_t)\mapsto(tX_{1/t})$. Since translations and reflection
generate also the symmetry group of the Askey--Wilson polynomials
\cite{MasatoshiStokman04}, it is natural to investigate how to
relate the measures of orthogonality of the Askey--Wilson polynomials
to quadratic harnesses.
The goal of this paper is to explore this idea and significantly
enlarge the class of
available examples. We show that quadratic harnesses exist and are
Markov processes for a wide range of parameters
$(\eta,\theta,\sigma,\tau,\gamma)$.
The basic Markov process we construct has Askey--Wilson polynomials as
orthogonal martingale polynomials. The final quadratic harness is
then obtained by
appropriate scaling and a deterministic change of time.
%
%%\comment{Should we move this down to "Conclusions"?}
%Since Askey-Wilson polynomials together with their limits are regarded
%as the most general classical polynomials, it is natural to inquire
%whether all quadratic harnesses can be obtained by taking limits of
%processes with Askey-Wilson transition probabilities.
%
\begin{theorem}
\label{Thm-QH-q}
Fix parameters $-1<q<1$ and $A,B,C,D$ that are either real, or
$(A,B)$ or $(C,D)$ are complex conjugate pairs
such that $ABCD$, $qABCD<1$. Assume that
%
%e1.4 ###
%
\begin{equation}
\label{ABCD}AC, AD, BC,BD, qAC, qAD, qBC,qBD \in{\mathbb C}\setminus
[1,\infty) .
\end{equation}
Let
%
%e1.9 ###
%e1.8 ###
%e1.7 ###
%e1.6 ###
%e1.5 ###
%
\begin{eqnarray}
\label{eta}
\eta&=&-\frac{ [(A+B)(1+ABCD)-2AB(C+D) ] \sqrt{1-q}}{\sqrt{(1-A C)
(1-B C) (1-A D) (1-B D) (1- q A B C D)}} ,\\
\label{theta}
\theta&=&-\frac{ [(D+C)(1+ABCD) -2 C D(A +B) ] \sqrt{1-q}}{\sqrt
{(1-A C) (1-B C) (1-A D) (1-B D) (1-q A B C D)}} ,\\
\label{sigma}
\sigma&=&
\frac{A B (1-q)}{1-q A B C D},
\\
\label{tau}
\tau&=& \frac{C D (1-q)}{1- q A B C D},
\\
\label{gamma}
\gamma&=& \frac{q-A B C D}{1- q A B C D}.
\end{eqnarray}
With the convention $1/\infty=0$, let
%
%e1.11 ###
%e1.10 ###
%
\begin{eqnarray}
\label{T0-greek}
T_0&=& \max
\biggl\{0,\frac{\gamma-1+\sqrt{(\gamma-1)^2-4 \sigma\tau
}}{2 \sigma},-\tau\biggr\},\\
\label{T1-greek}
\frac{1}{ T_1}&=&\max\biggl\{0,\frac{\gamma-1+\sqrt{(\gamma-1)^2-4
\sigma\tau
}}{2 \tau},-\sigma\biggr\}.
\end{eqnarray}

Then there exists a bounded Markov process $(X_t)_{t\in J}$ on the
nonempty interval $J=(T_0,T_1)$
with mean and covariance (\ref{EQ:cov}) such that
(\ref{EQ:LR}) holds,
and (\ref{EQ:q-Var}) holds with parameters $\eta,\theta,\sigma
,\tau,\gamma$. Process
$(X_t)_{t\in J}$ is unique among the processes with
infinitely-supported one-dimensional distributions that have moments of
all orders
and satisfy (\ref{EQ:cov}), (\ref{EQ:LR})
and (\ref{EQ:q-Var}) with the same parameters, $\eta,\theta,\sigma
,\tau,\gamma$.
\end{theorem}
\begin{remark} Formula (\ref{Def:X}) relates process $(X_t)$ to the
Markov process $(Z_t)$ from
Theorem \ref{T3}.
\end{remark}
\begin{remark}
The assumptions on $A,B,C,D$ are dictated by the desire to limit the
number of cases in the proof but do not exhaust all possibilities
where the quadratic harness $(X_t)$ with Askey--Wilson transition
probabilities exists (see Proposition \ref{T2}).
\end{remark}
\begin{remark}\label{R1.1}When $\sigma\tau\geq0$,
Theorem \ref{Thm-QH-q} can be used to construct quadratic harnesses
only for parameters in the range $-1<\gamma<1-2\sqrt{\sigma\tau}$
which is strictly smaller than the admissible range in \cite
{BrycMatysiakWesolowski04},
Theorem 2.2. To see the upper bound, note that
%
%e1.12 ###
%
\begin{equation}
\label{gamma-1}
1-\gamma=(1-q)(1-ABCD)/(1-qABCD)>0
\end{equation}
and that
$(1-\gamma)^2-4 \sigma
\tau=(1-ABCD)^2(1-q)^2/(1-qABCD)^2>0$. The lower bound follows from
$q>-1$, as (\ref{gamma}) defines $\gamma$ as an increasing function
of $q$.
In Corollary \ref{Th_PQH} we show that the construction indeed works
through the entire range of $\gamma$, at least when $\eta=\theta=0$.

From (\ref{gamma}), the construction will give $\gamma>1$ when
$ABCD<-1$. Multiplying (\ref{sigma}) and (\ref{tau}), we see that
this may occur only when $\sigma\tau<0$, that is,
%for harnesses on finite intervals
when the time interval $J$ is a proper subset of $(0,\infty)$ (compare
\cite{BrycMatysiakWesolowski04}, Theorem 2.2).
\end{remark}
%
%In Theorem \ref{Thm-QH-q}, $(\gamma-1)^2>4 \sigma\tau$. To see
%this, write $(\gamma-1)^2-4 \sigma
%in \cite[Theorem 2.2]{BrycMatysiakWesolowski04} this shows that
%to construct a quadratic harness on $[0,\infty)$ from Theorem
%
\begin{remark}\label{R1.2}
In terms of the original parameters, the
end-points of the interval, are
%
%e1.14 ###
%e1.13 ###
%
\begin{eqnarray}
\label{T0}
T_0&=& \max\biggl\{0, -CD, \frac{-CD(1-q)}{1-qABCD} \biggr\} ,\\
\label{T1}
\frac{ 1}{T_1}&=&\max\biggl\{ 0, -AB, \frac{-AB(1-q)}{1-qABCD} \biggr\} .
\end{eqnarray}
This shows that $T_0,T_1$ are real and $T_0<T_1$. If $CD<0$ or $AB<0$,
then the third term under the maximum contributes for $q<0$ only.
\end{remark}
\begin{remark}\label{R1.3}
As a side property, we also get information about one-sided conditioning:
${\mathbb E}(X_t|\mathcal{F}_s)=X_s$ and $\operatorname
{Var}(X_t|\mathcal{F}_s)=\frac
{t-s}{1+\sigma s}(1+\eta X_s+\sigma X_s^2)$ for $s<t$.
Similarly,
${\mathbb E}(X_t|\mathcal{F}_{\geq u})=t X_u/u$ and $\operatorname
{Var}(X_t|\mathcal
{F}_{\geq u})=\frac{t(u-t)}{u+\tau}(1+\theta X_u/u+\tau X_u^2/u^2)$
for $t<u$, where $\mathcal{F}_{\geq u}=\sigma(X_r\dvtx r\geq u)$.
\end{remark}

%From \eqref{gamma}, the construction will give $\gamma>1$ when
%$ABCD<-1$. Multiplying \eqref{sigma} and \eqref{tau}, we see that this
%may occur only when $\sigma\tau<0$, i.e.
%%for harnesses on finite intervals
%when the time interval $J$ is a proper subset of $(0,\infty)$; compare
% \cite[Theorem 2.2]{BrycMatysiakWesolowski04}.

%s1.2 ###
\subsection{Martingale property of Askey--Wilson polynomials}
\label{S: AW}
For $a,b,c,d\in
\mathbb{C}$ such that
%
%e1.15 ###
%
\begin{equation}
\label{abcd0}
abcd, qabcd\notin[1,\infty),
\end{equation}
Askey and Wilson \cite{AskeyWilson85}, (1.24), introduced polynomials
defined by recurrence,
%
%e1.16 ###
%
\begin{equation}\label{AW-proper}
2x
\widetilde{w}_n(x)= \widetilde A_n\widetilde{w}_{n+1}(x)+
{B}_n\widetilde{w}_n(x)+
\widetilde C_n \widetilde{w}_{n-1}(x),\qquad n\geq0,
\end{equation}
with the initial conditions $ \widetilde{w}_{-1}=0$ and $ \widetilde{w}_0=1$,
and with the coefficients
\begin{eqnarray*}
\widetilde A_n &=& \frac{ A_n}{(1-ab q^n)(1-acq^n)(1-adq^n)},
\\
B_n &=& a+1/a -A_n/a-aC_n,
\\
\widetilde C_n &=& C_n (1-ab q^{n-1})(1-acq^{n-1})(1-adq^{n-1}),
\end{eqnarray*}
where for future reference we denote
%
%e1.18 ###
%e1.17 ###
%
\begin{eqnarray}\label{A_n}
A_n &=& \frac{(1-abcd q^{n-1})(1-ab q^n)(1-ac q^n)(1-ad q^n)}{(1-abcd
q^{2n-1})(1-abcd q^{2n})},
\\
%
% \label{B_n}
% B_n=a+1/a -A_n/a-aC_n
%
\label{C_n}
C_n &=& \frac{ (1-q^n) (1-bc q^{n-1})(1-bd q^{n-1})(1-cd q^{n-1})}{
(1-abcd q^{2n-2})(1-abcd q^{2n-1})}.
\end{eqnarray}
Here we take $A_0=(1-ab)(1-ac)(1-ad)/(1-abcd)$ and $C_0=0$, also if $q=0$.
We remark that $B_n$ coincides with \cite{AskeyWilson85}, (1.27), so
it is symmetric in $a,b,c,d$ and that by taking the limit, $B_n$ is
also well defined for $a=0$.
Since
trivially
$\widetilde{A}_n$ and $\widetilde{C}_n$ are also symmetric in
$a,b,c,d$ it
follows
that polynomials $\{\widetilde w_n\}$ do not depend on the order of $a,b,c,d$.

% and $\widetilde A_n =\bar A_n/(1-ab q^n)$, $\widetilde B_n=B_n$,
% $\widetilde C_n=\bar C_n (1-ab q^{n-1})$.

%They do not depend on the order of $a,b,c,d$.

Except for Section \ref{S:discrete}, our parameters satisfy a
condition stronger than (\ref{abcd0}):
%
%e1.19 ###
%
\begin{equation}\label{abacad}
abcd,qabcd,ab,qab,ac,qac,ad,qad\in\mathbb{C}\setminus[1,\infty).
\end{equation}
To avoid cumbersome scaling of martingale polynomials later on, when
(\ref{abacad}) holds it is convenient to renormalize the polynomials
$\widetilde w_n$.
Therefore we
introduce the following family of polynomials:
%
%e1.20 ###
%
\begin{equation}
\label{AW}
%2x
2x\bar{w}_n(x)= \bar{A}_n\bar{w}_{n+1}(x)+ B_n\bar{w}_n(x)+
\bar{C}_n \bar{w}_{n-1}(x),\qquad n\geq 0,
\end{equation}
where $\bar A_n=(1-ab q^n) \widetilde A_n$,
$\bar C_n=\widetilde C_n/(1-ab q^{n-1})$.
%q^{n-1}).
The
initial conditions are again $\bar{w}_{-1}=0$ and $\bar{w}_0=1$.
When we want to indicate the parameters, we will write $\bar{w}_n(x;a,b,c,d)$.

For each $n$, polynomial $ \bar w_n $ differs only by a multiplicative
constant from $\widetilde w_n$ [see (\ref{AW2AW})]
so both families have the same orthogonality measure when it
exists. For this reason, both families of polynomials are referred to
as Askey--Wilson polynomials.

Recall that the polynomials $\{r_n(x;t)\dvtx n\in\mathbb{Z}_+, t\in I\}$ are
orthogonal martingale
polynomials for the process $(Z_t)_{t\in I}$ if:
\begin{longlist}
\item ${\mathbb E}(r_n(Z_t;t)r_m(Z_t;t) )=0$ for $m\ne n$ and
$t\in I$,
\item ${\mathbb E}(r_n(Z_t;t)|\mathcal{F}_s )=r_n(Z_s;s)$ for
$s<t$ in
$I$ and all $n=0,1,2\ldots.$
\end{longlist}

The following result shows that Askey--Wilson polynomials define
orthogonal martingale polynomials for a family of Markov processes.
\begin{theorem}\label{T3} Suppose that $A,B,C,D$ satisfy the
assumptions of Theorem~\ref{Thm-QH-q}. % \fbox{or Proposition \ref{T2}}.
Let
%
%e1.21 ###
%
\begin{equation}\label{Imax}\quad
I=I(A,B,C,D,q)= \biggl(\max\{0, CD, qCD\}, \frac{1}{\max\{0,AB,qAB\}} \biggr)
\end{equation}
with the convention $1/0=\infty$. (The last terms under the maxima can
contribute only when $q<0$ and $CD$ or $AB$ are negative.)
Let %\comment{revised!!}
%
%e1.22 ###
%
\begin{equation}\label{r_n}
r_n(x;t)=
t^{n/2}\bar w_n \biggl(\frac{\sqrt{1-q}}{2\sqrt{t}} x;A\sqrt{t},B\sqrt
{t},C/\sqrt{t},D/\sqrt{t} \biggr) .
\end{equation}
Then
\[
\{ r_n(x;t)\dvtx n=0,1,2\ldots, t\in I\}
\]
are orthogonal martingale polynomials for a Markov process $(Z_t)$
which satisfies (\ref{EQ:LR}) and (\ref{EQ:q-Var}) with $\eta
=\theta=\sigma=\tau=0$ and $\gamma=q$.
\end{theorem}
%
%$B$ as the distinguished first parameter in $\bar w_n$. The
%corresponding polynomials $r_n$ differ by the factor $
%both sets of polynomials are orthogonal martingale polynomials for the
%same process $(Z_t)$. (This follows from \cite[Theorem
%12.4.1]{Ismail05}.)

%s2 ###
\section{The case of densities}\label{S:elementary}
In this section we give an explicit and elementary construction of a
quadratic harness
%from Theorem \ref{Thm-QH-q}
on a possibly restricted time interval and under additional
restrictions on parameters $A,B,C,D$. %Later, in Section
%construction does not use orthogonal polynomials.
%
\begin{proposition}\label{Thm-QH-small}
Fix parameters $-1<q<1$ and $A,B,C,D$ that are either real or $(A,B)$
or $(C,D)$ are complex conjugate pairs. Without loss of generality, we
assume that
$|A|\leq|B|$, $ |C|\leq|D|$; additionally, we assume that $|B D|<1$.
Then the interval
\[
J= \biggl(\frac{|D|^2-CD}{1-AB |D|^2},\frac{1-CD|B|^2}{|B|^2-AB} \biggr)
\]
has positive length and there exists a unique bounded Markov process
$(X_t)_{t\in J}$ with absolutely continuous finite-dimensional distributions
which satisfies the conclusion of Theorem \ref{Thm-QH-q}.
%
%with mean and covariance \eqref{EQ: cov} such that
% and \eqref{EQ: q-Var} holds with parameters
% \eta&=&-\frac{((A+B)(1+ABCD)-2AB(C+D)) \sqrt{1-q}}{\sqrt{(1-A C) (1-B
%C) (1-A D) (1-B D) (1- q A B C D)}}\label{eta}\\
% \theta&=&-\frac{((D+C)(1+ABCD) -2 C D(A +B))) \sqrt{1-q}}{\sqrt{(1-A
%C) (1-B C) (1-A D) (1-B D) (1-q A B C D)}}\label{theta}\\
% \frac{A B (1-q)}{1-q A B C D}\label{sigma}
% \\
% \\
% \gamma&=& \frac{q-A B C D}{1- q A B C D}\label{gamma}
% \end{eqnarray}
\end{proposition}
\begin{remark}
Proposition \ref{Thm-QH-small} is a special case of Theorem \ref
{Thm-QH-q}; the latter may allow us to extend the processes constructed
here to a wider time-interval,
\[
T_0\leq\frac{|D|^2-CD}{1-AB |D|^2} \quad\mbox{and}\quad \frac
{1-CD|B|^2}{|B|^2-AB}\leq T_1.
\]
The easiest way to see the inequalities is to compare the end points of
intervals %$(|D|^2,1/|B|^2)$ , $(|CD|,1/|AB|)$
(\ref{I(A,B,C,D)}) and (\ref{Imax})
after the M\"{o}bius transformation (\ref{hhhh}); for example,
$|D|^2\geq|CD|\geq\max\{0, CD, qCD\}$ where the last term plays a
role only when $q<0$ and $CD<0$.
\end{remark}

The rest of this section contains the construction, ending with the
proof of Proposition \ref{Thm-QH-small}.

%s2.1 ###
\subsection{Askey--Wilson densities}\label{Sect:AWdens}
For complex $a$ and $|q|<1$ we define
%$$(a)_k=(a;q)_k= \{\begin{array}{ll} 1, & k=0, \\
%and
%
\begin{eqnarray*}
(a)_{n}&=&(a;q)_{n}=\cases{
\displaystyle\prod_{j=0}^{n-1} (1-aq^j), &\quad $n=1,2,\ldots,$\cr
1, &\quad $n=0$,}
\\
(a)_{\infty}&=&(a;q)_{\infty}=\prod_{j=0}^{\infty} (1-aq^j),
\end{eqnarray*}
%
%denote $q$-Pochhammer symbols.
and we denote
\begin{eqnarray*}
(a_1,a_2,\ldots,a_l)_{\infty} &=& (a_1,a_2,\ldots,a_l;q)_{\infty
}=(a_1;q)_{\infty}(a_2;q)_{\infty}\cdots(a_l;q)_{\infty} ,
\\
(a_1,a_2,\ldots,a_l)_{n} &=& (a_1,a_2,\ldots
,a_l;q)_{n}=(a_1;q)_{n}(a_2;q)_{n}\cdots(a_l;q)_{n} .
\end{eqnarray*}
The advantage of this notation over the standard product notation lies
both in its conciseness and in
mnemonic simplification rules,
%
%e2.1 ###
%
\begin{eqnarray}
\frac{(a, b)_n}{(a,c)_n} &=& \frac{(b)_n}{(c)_n} ,
\nonumber\\
\label{poch1}
(\alpha)_{M+L} &=& (q^M\alpha)_L (\alpha)_M
\end{eqnarray}
and
%
%e2.2 ###
%
\begin{equation}\label{poch2}
(\alpha)_M=(-\alpha)^Mq^{{M(M-1)}/{2}} \biggl(\frac{q}{q^M\alpha}
\biggr)_M,
\end{equation}
%
%$$
%$$
%$$(aq;q)_n=(-a)^n q^{-n(n-1)/2}(1/a;q)_n $$
which often help with calculations. For a reader who is
as uncomfortable with this notation, as we were at the beginning of
this project, we suggest to re-write the formulas for the
case $q=0$. For example, $(a;0)_n$ is either $1$ or $1-a$ as
$n=0$ or $n>0$, respectively. The construction of Markov process for
$q=0$ in itself is quite interesting as the resulting laws are
related to the laws that arise in Voiculescu's free probability; the
formulas simplify enough so that the integrals can be computed by
elementary means, for example, by residua.

From Askey and Wilson \cite{AskeyWilson85}, Theorem 2.1,
it follows that if $a,b,c,d$ are complex such that $\max\{
|a|,|b|,|c|,|d|\}<1$ and $-1<q<1$,
%$$a^2,ab, ac, ad, b^2,bc,bd,c^2,cd,d^2\not\in\{1,1/q,1/q^2,\ldots\},$$
then with $\theta=\theta_x$ such that
$\cos\theta=x$,
%
%e2.3 ###
%
\begin{eqnarray}\label{awint}\quad
&&\int_{-1}^1 \frac{1}{\sqrt{1-x^2}}\frac{ (e^{2i\theta},
e^{-2i\theta} )_{\infty}}
{ (ae^{i\theta},ae^{-i\theta}, be^{i\theta}, be^{-i\theta},
ce^{i\theta}, ce^{-i\theta}, de^{i\theta}, de^{-i\theta} )_{\infty
}} \,dx\nonumber\\[-8pt]\\[-8pt]
&&\qquad=\frac{2\pi(abcd)_{\infty}}{(q, ab, ac, ad, bc, bd, cd)_{\infty}}
.\nonumber
\end{eqnarray}
When $-1<q<1$ and $a,b,c,d$ are either real
or come in complex conjugate pairs and
$\max\{|a|,|b|,|c|,|d|\}<1$, the integrand is real and positive.
This allows us to define the Askey--Wilson density,
%
%e2.4 ###
%
\begin{equation}\label{AWdistr}\qquad
f(x;a,b,c,d)=\frac{K(a,b,c,d)}{\sqrt{1-x^2}} \biggl|\frac{(e^{2i\theta
})_{\infty}}{(ae^{i\theta}, be^{i\theta}, ce^{i\theta},
de^{i\theta})_{\infty}} \biggr|^2I_{(-1,1)}(x),
\end{equation}
where
%
%e2.5 ###
%
\begin{equation}\label{K}
K(a,b,c,d)=\frac{(q, ab, ac, ad, bc, bd, cd)_{\infty}}{2\pi
(abcd)_{\infty}} .
\end{equation}
The first two moments are easily computed.
\begin{proposition}\label{awmom}
Suppose $X$ has the Askey--Wilson density $f(x;a,b,c,d)$ with
parameters $a,b,c,d$ as above. Then the expectation of $X$ is
%
%e2.6 ###
%
\begin{equation}\label{1mom}
{\mathbb E}(X)=\frac{a+b+c+d-abc-abd-acd-bcd}{2(1-abcd)}
\end{equation}
and the variance of $X$ is
%
%e2.7 ###
%
\begin{equation}\label{awvar}\qquad
\operatorname{Var}(X)=\frac
{(1-ab)(1-ac)(1-ad)(1-bc)(1-bd)(1-cd)(1-q)}{4(1-abcd)^2(1-abcdq)} .
\end{equation}
\end{proposition}
%
%We remark that since we could have used orthogonal polynomials to
%compute the moments,
% the same answer holds for the more general Askey-Wilson distribution
%with discrete component.
%
\begin{pf}
If $a=b=c=d=0$, ${\mathbb E}(X)=0$ by symmetry. If one of the
parameters, say
$a\in{\mathbb C}$, is nonzero, we note that
\begin{eqnarray*}
(ae^{i\theta},
a e^{-i\theta} )_{\infty} &=& (ae^{i\theta},
a e^{-i\theta} )_{1} (aqe^{i\theta},aqe^{-i\theta} )_{\infty} \\
&=& (1+a^2-2ax) (aqe^{i\theta},aqe^{-i\theta} )_{\infty} .
\end{eqnarray*}
Therefore, by (\ref{awint}),
\[
{\mathbb E}(1+a^2-2aX)=\frac{K(a,b,c,d)}{K(qa,b,c,d)}=\frac
{(1-ab)(1-ac)(1-ad)}{1-abcd} .
\]
Now (\ref{1mom}) follows by a simple algebra.

Similarly, for nonzero $a,b\in{\mathbb C}$,
\begin{eqnarray*}
&&\hspace*{-31.25pt}4ab \operatorname{Var}
(X)={\mathbb E}[(1+a^2-2aX)(1+b^2-2bX) ],
\\
&&-{\mathbb E}(1+a^2-2aX)
{\mathbb E}(1+b^2-2bX)\\
&&\qquad=\frac{K(a,b,c,d)}{K(qa,qb,c,d)}
-\frac{K^2(a,b,c,d)}{K(qa,b,c,d)K(a,qb,c,d)}
\\
&&\qquad=\frac{(1-ab)(1-qab)(1-ac)(1-ad)(1-bc)(1-bd)}{(1-abcd)(1-qabcd)}\\
&&\qquad\quad{}-\frac{(1-ab)^2(1-ac)(1-ad)(1-bc)(1-bd)}{(1-abcd)^2}.
\end{eqnarray*}
Again after simple transformations we arrive at (\ref{awvar}).

If only one parameter is nonzero but $q\ne0$, the calculations are
similar, starting with
${\mathbb E}((1+a^2-2aX)(1+a^2q^2-2aqX) )$; when $q=0$ the density
is a re-parametrization of Marchenko--Pastur law
\cite{HiaiPetz00}, (3.3.2); we omit the details. If $a,b,c,d$ are
zero, $f(x;0,0,0,0)$ is the
orthogonality measure of the continuous $q$-Hermite polynomials
\cite{KoekoekSwarttouw}, (3.26.3); since $H_2(x)=2 x
H_1(x)-(1-q)H_0=4x^2-(1-q)$, the second moment is
$(1-q)/4$.
\end{pf}

We need a
technical result on Askey--Wilson densities inspired by \cite
{NassrallahRahman85},
formula~(2.4).
\begin{proposition}\label{techlem}
Let $a,b,c,d,q$ be as above with the additional assumption that
the only admissible conjugate pairs are $a=\bar{b}$ or
$c=\bar{d}$, and $m$ is real such that $|m|<1$. Then with $x=\cos
\theta_x$,
%
%e2.8 ###
%
\begin{eqnarray}\label{inteq}
&&\int_{-1}^1 f (x;am,bm,c,d )f (y;a,b,me^{i\theta_x},me^{-i\theta_x})
\,dx \nonumber\\[-8pt]\\[-8pt]
&&\qquad= f (y;a,b,cm,dm ).\nonumber%\hspace*{-32pt}
\end{eqnarray}
\end{proposition}
\begin{pf}
We compute the left-hand side of (\ref{inteq}) expanding the constants
$K(am,bm,c,d)$ and $K(a,b,me^{i\theta_x},me^{-i\theta_x})$ to
better show how some factors cancel out. To avoid case-by-case
reasoning when complex conjugate pairs are present, we also expand
parts of the density
without the use of modulus as in (\ref{awint}).

The integrand on the left-hand side of (\ref{inteq}) is
\begin{eqnarray*}
&&\frac{(q, abm^2, acm, adm, bcm, bdm, cd)_{\infty} | (e^{2i\theta_x}
)_{\infty}|^2}
{2\pi(abcdm^2)_{\infty} (ame^{i\theta_x}, ame^{-i\theta_x},
bme^{i\theta_x}, bme^{-i\theta_x} )_\infty| (ce^{i\theta_x},
de^{i\theta_x} )_{\infty}|^2}\\
&&\qquad{}\times\frac{ (q, ab, ame^{i\theta_x}, ame^{-i\theta_x},
bme^{i\theta_x}, bme^{-i\theta_x}, m^2 )_{\infty}}
{2\pi(abm^2)_{\infty}\sqrt{1-y^2}}
\\
&&\qquad{}\times\frac{| (e^{2i\theta_y} )_{\infty}|^2}{| (ae^{i\theta_y},
be^{i\theta_y},
me^{i(\theta_x+\theta_y)}, me^{i(-\theta_x+\theta_y)} )_{\infty
}|^2\sqrt{1-x^2}} .
\end{eqnarray*}
Rearranging the terms we rewrite the left-hand side of (\ref{inteq}) as
\begin{eqnarray*}
&&\frac{(q, abm^2, acm, adm, bcm, bdm, cd, q, ab, m^2)_{\infty}}{(2\pi
)^2(abm^2, abcdm^2)_{\infty}\sqrt{1-y^2}}
\\
&&\qquad{}
\times\frac{| (e^{2i\theta_y} )_{\infty}|^2}{| (ae^{i\theta_y},
be^{i\theta_y} )_{\infty}|^2}
\\
&&\qquad{}\times\int_{-1}^1 \frac{| (e^{2i\theta_x} )_{\infty}|^2}
{| (me^{i\theta_y}e^{i\theta_x}, me^{-i\theta_y}e^{i\theta_x},
ce^{i\theta_x}, de^{i\theta_x} )_{\infty}|^2}
\frac{dx}{\sqrt{1-x^2}} .
\end{eqnarray*}
Now we apply formula (\ref{awint}) to this integral, so the left-hand
side of (\ref{inteq}) becomes
\begin{eqnarray*}
&&\frac{(q, abm^2, acm, adm, bcm, bdm, cd, q, ab, m^2)_{\infty}}{(2\pi
)^2(abm^2, abcdm^2)_{\infty}\sqrt{1-y^2}}
\times\frac{| (e^{2i\theta_y} )_{\infty}|^2}{| (ae^{i\theta_y},
be^{i\theta_y} )_{\infty}|^2}
\\
&&\quad{}\times\frac{2\pi(cdm^2)_{\infty}}{(q, m^2, mce^{i\theta_y},
mde^{i\theta_y}, mce^{-i\theta_y}
, mde^{-i\theta_y}, cd)_{\infty}}
\\
&&
\qquad
=\frac{(q, ab, acm, adm, bcm, bdm, cdm^2)_{\infty} | (e^{2i\theta_y}
)_{\infty} |^2}
{2\pi(abcdm^2)_{\infty}| (ae^{i\theta_y}, be^{i\theta_y}
, mce^{i\theta_y}, mde^{i\theta_y} )_{\infty}|^2\sqrt{1-y^2}},
\end{eqnarray*}
which completes the proof.
\end{pf}

%s2.2 ###
\subsection{Markov processes with Askey--Wilson densities}

We now fix $A$, $B$, $C$, $D$ as in Proposition \ref{Thm-QH-small}.
The interval
%
%e2.9 ###
%
\begin{equation}
\label{I(A,B,C,D)}
I(A,B,C,D)= \biggl(|D|^2, \frac{1}{|B|^2} \biggr)
\end{equation}
is nonempty (here $1/0=\infty$).
For any $t\in I(A,B,C,D)$ and $y\in[-1,1]$ let
%
%e2.10 ###
%
\begin{equation}\label{marg-d}
p(t,y)=f \biggl(y;A\sqrt{t},B\sqrt{t},\frac{C}{\sqrt{t}},\frac{D}{\sqrt
{t}} \biggr)
\end{equation}
and for any $s<t$ in $I(A,B,C,D)$ and $x,y\in[-1,1]$ let
%p(s,x;t,y)=f (y;A\sqrt{t},B\sqrt{t},\sqrt{\frac{s}{t}}e^{i\theta_x},
% x=\cos\theta_x
%
%e2.11 ###
%
\begin{equation}\label{trans-d}\qquad
p(s,x;t,y)=f \biggl(y;A\sqrt{t},B\sqrt{t},\frac{\sqrt{s}}{\sqrt
{t}}e^{i\theta_x},\frac{\sqrt{s}}{\sqrt{t}}e^{-i\theta_x}
\biggr),\qquad
x=\cos\theta_x .
\end{equation}
\begin{proposition}\label{mf}
The family of probability
densities $(p(s,x;t,y), p(t,y))$ defines a Markov process $(Y_t)_{t\in
I}$ on the state space $[-1,1]$.
That is, for any $s<t$ from $I(A,B,C,D)$ and $y\in[-1,1]$,
%
%e2.12 ###
%
\begin{equation}\label{marpro1}
p(t,y)=\int_{-1}^1 p(s,x;t,y)p(s,x) \,dx
\end{equation}
and for any $s<t<u$ from $I(A,B,C,D)$ and $x,z\in[-1,1]$,
%
%e2.13 ###
%
\begin{equation}\label{marpro2}
p(s,x;u,z)=\int_{-1}^1 p(t,y;u,z)p(s,x;t,y)\, dy .
\end{equation}
\end{proposition}
\begin{pf}
To show (\ref{marpro1}) it suffices just to use the identity
(\ref{inteq})
with $a=
A\sqrt{t}$, $b= B\sqrt{t}$, $c= C/\sqrt{s}$, $d=
D/\sqrt{s}$, and $m=\sqrt{s/t}\in(0,1)$.
We note that this substitution preserves the complex conjugate pairs
and that, by the definition of $I(A,B,C,D)$, parameters
$A\sqrt{t}$, $B\sqrt{t}$, $C/\sqrt{s}$ and $D/\sqrt{s}$ have
modulus less than one. So (\ref{inteq}) applies here and gives the
desired formula,
\begin{eqnarray*}
&&\int_{-1}^1 f \biggl(x;A\sqrt{s},B\sqrt{s},\frac{C}{\sqrt{s}},\frac
{D}{\sqrt{s}} \biggr)
f \biggl(y;A\sqrt{t},B\sqrt{t},\frac{\sqrt{s}}{\sqrt{t}}e^{i\theta
_x},\frac{\sqrt{s}}{\sqrt{t}}e^{-i\theta_x} \biggr) \,dx\\
&&\qquad={f} \biggl(y;A\sqrt{t},B\sqrt{t},\frac{C}{\sqrt{t}},\frac{D}{\sqrt
{t}} \biggr) .
\end{eqnarray*}
To get the second formula (\ref{marpro2}) we again use
(\ref{inteq})
this time with $a= A\sqrt{u}$, $b= \sqrt{u}$,
$c=\sqrt{s/t}e^{i\theta_x}$, $d=\sqrt{s/t}e^{-i\theta_x}$,
and $m=\sqrt{t/u}$. Thus we arrive at
\begin{eqnarray*}
&&\int_{-1}^1f \biggl(z;A\sqrt{u},B\sqrt{u},\frac{\sqrt{t}}{\sqrt
{u}}e^{i\theta_y},\frac{\sqrt{t}}{\sqrt{u}}e^{-i\theta_y} \biggr)
\\
&&\quad{}\times f \biggl(y;A\sqrt{t},B\sqrt{t},\frac{\sqrt{s}}{\sqrt
{t}}e^{i\theta_x},\frac{\sqrt{s}}{\sqrt{t}}e^{-i\theta_x} \biggr) \,dy
\\
&&\qquad={f} \biggl(z;A\sqrt{u},B\sqrt{u},\frac{\sqrt{s}}{\sqrt{u}}e^{i\theta
_x},\frac{\sqrt{s}}{\sqrt{u}}e^{-i\theta_x} \biggr).
\end{eqnarray*}
\upqed\end{pf}
\begin{proposition}\label{condmom}
Let $(Y_t)_{t\in I(A,B,C,D)}$ be the Markov process from Proposition
\ref{mf}, with marginal densities (\ref{marg-d}) and transition
densities (\ref{trans-d}).
For $t\in I(A,B,C,D)$,
%
%e2.15 ###
%e2.14 ###
%
\begin{eqnarray}\label{ordmom1}
{\mathbb E}(Y_t) &=& \frac{[A+B-AB(C+D)]t+C+D-CD(A+B)}{2\sqrt
{t}(1-ABCD)} ,
\\
\label{ordvar}
\operatorname{Var}(Y_t)
&=&\frac{(1-q)(1-AC)(1-AD)(1-BC)(1-BD)}{4t(1-ABCD)^2(1-qABCD)}\nonumber\\[-8pt]\\[-8pt]
&&{}\times(t-CD)(1-ABt)\nonumber
\end{eqnarray}
and for $s,t\in I(A,B,C,D)$, such that $s<t$,
%
%e2.18 ###
%e2.17 ###
%e2.16 ###
%
\begin{eqnarray}\label{covar}
\operatorname{Cov}(Y_s,Y_t)
&=&\frac{(1-q)(1-AC)(1-AD)(1-BC)(1-BD)}{4\sqrt
{st}(1-ABCD)^2(1-qABCD)}\nonumber\\[-8pt]\\[-8pt]
&&{}\times(s-CD)(1-ABt) ,\nonumber
\\
\label{condmom1}
{\mathbb E}(Y_t|{\mathcal
F}_s)&=&\frac{(A+B)(t-s)+2(1-ABt)\sqrt{s}Y_s}{2\sqrt{t}(1-ABs)} ,
\\
\label{condvar}
\operatorname{Var}(Y_t|{\mathcal
F}_s)&=&\frac{(1-q)(t-s)(1-ABt)}{4t(1-ABs)^2(1-qABs)}\bigl(1+A^2s-2A\sqrt
{s}Y_s\bigr)\nonumber\\[-8pt]\\[-8pt]
&&{}\times\bigl(1+B^2s-2B\sqrt{s}Y_s\bigr).\nonumber
\end{eqnarray}
\end{proposition}
\begin{pf}
Formulas (\ref{ordmom1}) and (\ref{ordvar}) follow,
respectively, from
(\ref{1mom}) and (\ref{awvar}) by taking $a= A\sqrt{t}$, $b=
B\sqrt{t}$, $c= C/\sqrt{t}$ and $d= D/\sqrt{t}$.

Similarly, the formulas (\ref{condmom1}) and (\ref{condvar})
follow, respectively, from
(\ref{1mom})
and (\ref{awvar}) by
taking $a= A\sqrt{t}$, $b=
B\sqrt{t}$, $c= \sqrt{\frac{s}{t}}e^{i\theta_x}$ and $d=
\sqrt{\frac{s}{t}}e^{-i\theta_x}$.

To obtain the covariance we make use of (\ref{condmom1}) as
follows:
\begin{eqnarray*}
\operatorname{Cov}(Y_s,Y_t) &=& {\mathbb E}(Y_s{\mathbb E}(Y_t|{\mathcal
F}_s))-{\mathbb E}(Y_s){\mathbb E}(Y_t)\\
&=& \biggl(\frac{(A+B)(t-s)}{2\sqrt
{t}(1-ABs)}-{\mathbb E}
Y_t \biggr){\mathbb E}Y_s\\
&&{} + \frac{(1-ABs)\sqrt{s}}{(1-ABt)\sqrt{t}} \bigl(\operatorname
{Var}(Y_s)+[{\mathbb E}
Y_s]^2 \bigr) .
\end{eqnarray*}

Now the formula (\ref{covar}) follows, after a calculation,
from (\ref{ordmom1}) and (\ref{ordvar}).
\end{pf}

%Next we consider conditional distributions of $Y_t$
%given the past and the future of the process.
Next we show that the conditional distribution of $Y_t$
given the past and the future of the process is given by an
Askey--Wilson density that
does not depend on parameters $A,B,C,D$.
\begin{proposition}\label{pfcond}
Let $(Y_t)_{t\in I(A,B,C,D)}$ be the Markov process with
marginal densities (\ref{marg-d}) and transition
densities (\ref{trans-d}). Then for
any $s<t<u $ in $I(A,B,C,D)$,
%$s,t,u\in I(A,B,C,D)$ such that $s<t<u $
the conditional
distribution of $Y_t$ given $\mathcal{F}_{s,u}$
has the Askey--Wilson density,
%f (y;\sqrt{\frac{t}{u}}(z+i\sqrt{1-z^2}),\sqrt{\frac{t}{u}}(z-i
%
%e2.19 ###
%
\begin{equation}\label{pfconddis}\quad
f \biggl(y;\frac{\sqrt{t}}{\sqrt{u}}\exp(i\theta_z),\frac{\sqrt
{t}}{\sqrt{u}}\exp(-i\theta_z),\frac{\sqrt{s}}{\sqrt{t}}\exp
(i\theta_x),
\frac{\sqrt{s}}{\sqrt{t}}\exp(-i\theta_x) \biggr) .
\end{equation}
(Here, $x=\cos\theta_x=Y_s$, $z=\cos\theta_z=Y_u$.)
The first two conditional moments have the form
%
%e2.21 ###
%e2.20 ###
%
\begin{eqnarray}
\label{pfcondexp}
{\mathbb E}(Y_t|{\mathcal
F}_{s,u}) &=& \frac{(u-t)\sqrt{s}Y_s+(t-s)\sqrt{u}Y_u}{\sqrt{t}(u-s)} ,
\\
\label{pfcondvar}
\operatorname{Var}(Y_t|{\mathcal
F}_{s,u}) &=& \frac{(1-q)(u-t)(t-s)}{t(u-qs)}\nonumber\\[-8pt]\\[-8pt]
&&{}\times \biggl(\frac{1}{4}
-\frac{(u\sqrt{s}Y_s-s\sqrt{u}Y_u)(\sqrt{u}Y_u-\sqrt
{s}Y_s)}{(u-s)^2} \biggr).\nonumber
\end{eqnarray}
\end{proposition}
\begin{pf}
By the Markov property it follows that the conditional density is
\begin{eqnarray*}
&&\frac{p(t,y;u,z) p(s,x;t,y)}{p(s,x;u,z)}\\
%=%WWW old incorrect expression
%f(y;A\sqrt{t},B\sqrt{t},\sqrt{s/t} e^{i\theta_x},\sqrt{s/t} e^{i
%{f(z;A\sqrt{u},B\sqrt{u},\sqrt{s/u} e^{i\theta_x},\sqrt{s/t} e^{-i
%
&&\qquad=f\biggl(z;A\sqrt{u},B\sqrt{u},\frac{\sqrt{t}}{\sqrt{u}}
e^{i\theta_y},\frac{\sqrt{t}}{\sqrt{u}} e^{-i\theta_y}\biggr)
\\
&&\qquad\quad{}\times
f\biggl(y;A\sqrt{t},B\sqrt{t},\frac{\sqrt{s}}{\sqrt{t}} e^{i\theta
_x},\frac{\sqrt{s}}{\sqrt{t}} e^{-i\theta_x}\biggr)\\
&&\qquad\quad{}\times\biggl({f\biggl(z;A\sqrt{u},B\sqrt{u},\frac{\sqrt{s}}{\sqrt{u}} e^{i\theta
_x},\frac{\sqrt{s}}{\sqrt{u}} e^{-i\theta_x}\biggr)}\biggr)^{-1}.
\end{eqnarray*}
Now the result follows by plugging in the formula above the
definition of the Askey--Wilson density
%WW \eqref{AWdistr}
(\ref{AWdistr})
with
suitably chosen parameters. The mean and variance are calculated from
(\ref{1mom}) and (\ref{awvar}).
\end{pf}
\begin{pf*}{Proof of Proposition \protect\ref{Thm-QH-small}}
If we define a new process $(Z_t)_{t\in I(A,B,C,D)}$ through
%
%e2.22 ###
%
\begin{equation}
\label{Z}
Z_t=\frac{2\sqrt{t}}{\sqrt{1-q}}Y_t ,
\end{equation}
then $(Z_t)$ is Markov and, for $s<t$, satisfies
\[
{\mathbb E}(Z_t|{\mathcal
F}_s)=\frac{(A+B)(t-s)}{\sqrt{1-q}(1-ABs)}+\frac{1-ABt}{1-ABs}Z_s
\]
so that
%$$
% (\frac{Z_t-\frac{A+B}{AB\sqrt{1-q}}}{1-ABt}, {\mathcal
%F}_t )
%$$
%
\[
\biggl(\frac{AB\sqrt{1-q}Z_t-{(A+B)}}{1-ABt}, {\mathcal
F}_t \biggr)
\]
is a martingale.
Moreover,
\begin{eqnarray*}
\operatorname{Var}(Z_t|{\mathcal F}_s)
&=&\frac{(t-s)(1-ABt)}{(1-ABs)^2(1-qABs)}\\
&&{}\times\bigl(1+A^2s-A\sqrt
{1-q}Z_s\bigr)\bigl(1+B^2s-B\sqrt{1-q}Z_s\bigr) .
\end{eqnarray*}

For the double conditioning with respect to the past and future
jointly, it follows that $(Z_t)$ satisfies quadratic harness
conditions; for $s<t<u$,
%
%e2.23 ###
%
\begin{equation}\label{Z-LR}
{\mathbb E}(Z_t|{\mathcal F}_{s,u})=\frac{u-t}{u-s}Z_s+\frac{t-s}{u-s}Z_u
\end{equation}
and
%
%e2.24 ###
%
\begin{equation}\label{Z:q-Var}\hspace*{28pt}
\operatorname{Var}(Z_t|{\mathcal
F}_{s,u})=\frac{(u-t)(t-s)}{u-qs} \biggl(1-(1-q)\frac
{(uZ_s-sZ_u)(Z_u-Z_s)}{(u-s)^2} \biggr),
\end{equation}
which correspond to the $q$-Brownian motion (see \cite{BrycWesolowski03},
Theorem 4.1).
Here, $(Z_t)$ is defined only on a possibly-bounded time
domain $I(A,B,C,D)$, and the covariance is different than in
\cite{BrycWesolowski03};
for $s<t$,
%
%e2.25 ###
%
\begin{eqnarray}\label{Z-cov}
\operatorname{Cov}(Z_s,Z_t)
&=&\frac
{(1-AC)(1-AD)(1-BC)(1-BD)}{(1-ABCD)^2(1-qABCD)}\nonumber\\[-8pt]\\[-8pt]
&&{}\times(s-CD)(1-ABt) .\nonumber
\end{eqnarray}
(The law of $Z_t$ will differ from the $q$-Gaussian law if $|A|+|B|+|C|+|D|>0$.)

%So a $q$-Gaussian non-commutative triple with the above covariance
%cannot have a classical version \cite{Bryc01} unless $A=B=C=D=0$

The covariance is adjusted by a suitable deterministic time change.
Consider a M\"{o}bius transformation
%
%e2.26 ###
%
\begin{equation}
\label{hhhh}
h(x)= \frac{x-CD}{1-ABx} ,
\end{equation}
which for $ABCD<1$ is an increasing function
with the inverse,
%
%e2.27 ###
%
\begin{equation}\label{T(t)}
T(t)=\frac{t+CD}{1+ABt} .
\end{equation}
Note that $J=J(A,B,C,D)=h(I(A,B,C,D))$.
For $t\in J(A,B,C,D)$, define
%
%e2.28 ###
%
\begin{eqnarray}\label{Def:X}\hspace*{28pt}
X_t:\!&=&X_{t;A,B,C,D,q} \nonumber\\
&=&
\frac{Z_{T(t)}-{\mathbb E}(Z_{T(t)})}{1-T(t)AB} \times\frac{(1-A B C D)
\sqrt{1-q A B C D}}{\sqrt{(1-AC) (1-BC)
(1-AD) (1-BD)}}\\
&=& \frac{\sqrt{1-q}(1+ABt)Z_{T(t)}-(A+B)t-(C+D)}{\sqrt
{(1-q)(1-AC)(1-AD)(1-BC)(1-BD)}}\sqrt{1-q A B C D}.\nonumber
\end{eqnarray}
A calculation shows that $(X_t)_{t\in J}$ has unconditional and
conditional moments as claimed:
formula (\ref{EQ:cov}) is a consequence of (\ref{Z-cov}), and (\ref
{EQ:LR}) follows from (\ref{Z-LR}). A much longer
calculation shows that (\ref{Z:q-Var}) translates into (\ref{EQ:q-Var})
with parameters (\ref{eta})--(\ref{gamma}).
\end{pf*}
%
%%%%%%%%%%%%%%%%
%%% discrete case was here
%%%%%%%%%%%%%%%%

%s3 ###
\section{Construction in the general case} \label{S:general}
Next, we tackle the issue of extending the quadratic harness from
Proposition \ref{Thm-QH-small}
to a larger time interval. The main technical difficulty is that such
processes may have a discrete
component in their distributions. The
construction is based on the Askey--Wilson distribution
\cite{AskeyWilson85}, (2.9), with slight correction as in
\cite{StokmanKoornwinder98}, (2.5).

The basic plan of the proof of Theorem \ref{Thm-QH-q} is the same as
that of the proof of Proposition \ref{Thm-QH-small}: we
define auxiliary Markov process $(Y_t)_{t\in I}$ through a family of
Askey--Wilson distributions that satisfy the Chapman--Kolmogorov
equations. Then we use formulas
(\ref{Z}) and (\ref{Def:X}) to define $(X_t)$. The main difference is
that due to an overwhelming
number of cases that arise with mixed-type distributions, we use
orthogonal polynomials to deduce all properties we need.
(A similar approach was used in \cite{BrycMatysiakWesolowski04b}.)
%%In particular, we conjecture that the conditional laws are always
%Askey-Wilson as in Proposition \ref{pfcond}, but we do not prove this
%fact in general.

%s3.1 ###
\subsection{The Askey--Wilson law}
%factor of $q^n(n+1)/2$.
The Askey--Wilson distribution $\nu(dx;a,b,c,d)$ is the
(probabilistic) orthogonality measure
of the Askey--Wilson polynomials $\{\widetilde w_n\}$ as defined in
(\ref{AW-proper}). Therefore it does not depend on the order of
parameters $a,b,c,d$. Since $ \{\widetilde A_n\}$, $\{\widetilde B_n\}$
and $\{\widetilde C_n\}$ are bounded sequences, $\nu(dx;a,b,c,d)$ is
unique and compactly supported \cite{Ismail05}, Theorems
2.5.4 and~2.5.5. %\cite[page ...]{Chihara78}.
If $|a|,|b|,|c|$, $|d|<1$, this is an absolutely continuous measure with
density (\ref{AWdistr}). For other values of parameters, $\nu
(dx;a,b,c,d)$ may have a discrete component or be purely discrete as in
(\ref{pmf}).

In general, it is quite difficult to give explicit conditions for the
existence of
the Askey--Wilson distribution $\nu(dx;a,b,c,d)$ in terms of $a,b,c,d$.
To find sufficient conditions, we will be working with sequences $\{
A_k\}$, and $ \{C_k\}$ defined by (\ref{A_n}) and (\ref{C_n}). Since
$\widetilde A_{k-1}\widetilde C_k=A_{k-1}C_k$, by
Theorem \ref{Thm-Favard},
measure $\nu(dx;a,b,c,d)$
exists for all $a,b,c,d$ such that sequences
$\{A_k\}$, $\{C_k\}$ are real, and (\ref{Farvard_condition}) holds.
%% Favard's theorem \cite[Theorem 2.5.2]{Ismail05}, %\cite[page
%...]{Chihara78},
% measure $\nu(dx;a,b,c,d)$
% exists for all $a,b,c,d$ such that $A_k,C_k\in\RR$ for all $k$ and
%$\prod_{k=0}^n A_k C_{k+1}\geq0$ for all $n$. Furthermore, it is
%known that if $N$ is the first integer such that $A_{N-1} C_{N}=0$,
%then
%$\nu(dx;a,b,c,d;q)$ is a discrete measure supported on the finite set
%of %WWW
%$N$ zeros of the polynomial
%$\bar w_{N}(x)$. %
%% (The assumption that $A_kC_{k+1}>0$ for all $k$, implies that the
%support is infinite but we cannot restrict ourselves to this case.)
If $a,b,c, d$ are either real or come in complex conjugate pairs and
(\ref{abacad}) holds,
then $A_k>0$ and $C_k\in\mathbb{R}$ for all $k$. So in this case condition,
(\ref{Farvard_condition})
%is
becomes
%
%e3.1 ###
%
\begin{equation}\label{Prod_C_k}
\prod_{k=1}^n C_k\geq0 \qquad\mbox{for all $n\geq1$}.
\end{equation}

A simple sufficient condition for (\ref{Prod_C_k}) is that in addition
to (\ref{abacad}) we have
%
%e3.2 ###
%
\begin{equation}\label{abcd}
bc,qbc,bd,qbd,cd,qcd\in{\mathbb C}\setminus[1,\infty) .
\end{equation}
Under this condition, if $a,b,c,d$ are either real or come in complex
conjugate pairs, then the corresponding measure of orthogonality $\nu
(dx;a,b,c,d)$ exists.
Unfortunately, this simple condition is not general enough for our
purposes; we need to allow also Askey--Wilson laws with finite support
as in \cite{AskeyWilson79}. In fact, such laws describe transitions
of the Markov process in the atomic
part.

We now state conditions that cover all the cases needed in this paper.
Let $m_1=m_1(a,b,c,d)$ denote the number of the products $ab, ac,ad,bc,bd,cd$
that fall into
subset $[1,\infty)$ of complex plane, and
let $m_2=m_2(a,b,c,d)$ denote the number of the products $qab$,
$qac$, $qad$, $qbc$, $qbd$, $qcd$ that fall into $[1,\infty)$.
(For $m_1=0$,
measure $\nu$ is described in \cite{Stokman1997}.)
\begin{lemma}\label{L-AW-criter}
Assume that $a,b,c, d$ are either real or come in complex conjugate
pairs and that $abcd<1$, $qabcd<1$.
%Only the following cases are possible:
Then the Askey--Wilson distribution $\nu$ exists only in the following cases:
\begin{longlist}
\item If $q\geq0$ and $m_1=0$, then $\nu(dx;a,b,c,d)$ exists
and has a continuous component.
\item If $q<0$ and $m_1=m_2=0$, then $\nu(dx;a,b,c,d)$ exists
and has a continuous component.
\item If $q\geq0$ and $m_1=2$, then $a,b,c,d\in\mathbb{R}$. In
this case, $\nu(dx;a,b,c,d)$ is well defined if either $q=0$ or the
smaller of the
two products that fall into $[1,\infty)$ is of the form $1/q^N$, and
in this latter case $\nu(dx;a,b,c,d)$ is a
purely discrete measure with $N+1$ atoms.
\item If $q<0$ and $m_1=2$, $m_2=0$ then $a,b,c,d\in\mathbb
{R}$. In
this case, $\nu(dx;a,b,\break c,d)$ is well defined if
the smaller of the two products in $[1,\infty)$ equals
$1/q^N$ with even $N$. Then $\nu(dx;a,b,c,d)$ is a
purely discrete measure with $N+1$ atoms.
\item If $q<0$, $m_1=0$ and $m_2=2$, then $a,b,c,d\in\mathbb
{R}$. In
this case, $\nu(dx;a,b,\break c,d)$ is well defined if
the smaller of the two products in $[1,\infty)$ equals
$1/q^N$ with even $N$. Then $\nu(dx;a,b,c,d)$ is a
purely discrete measure with $N+2$ atoms.
\end{longlist}
\end{lemma}
\begin{pf}
We first note that in order for $\nu$ to exist when
$q\geq0$, we must have
either $m_i=0$ or $m_i=2$, $i=1,2$. This is an elementary observation based on
the positivity of $A_0C_1$ and $A_1C_2$ [see (\ref{A_n}), (\ref{C_n})].

Similar considerations show that
if $q<0$ and $m_1m_2>0$, then (\ref{Farvard_condition}) fails, and
$\nu(dx;a,b,c,d)$ does not
exist. Furthermore, there are only three possible choices:
$(m_1,m_2)=(0,0),(0,2), (2,0)$.

%Existence of $\nu(dx;a,b,c,d)$ follows by inspection of $\prod_{k=1}^n
%A_{k-1}C_k$.
If $m_2>0$, then in cases (iii) and (iv), the product $\prod_{k=1}^n
A_{k-1}C_k>0$ for $n<N$ and is zero for $n\geq N+1$.
In case (v), $\prod_{k=1}^n A_{k-1}C_k=0$ for all $n\geq N+2$.
\end{pf}

According to Askey and Wilson \cite{AskeyWilson85} the orthogonality
law is
\[
\nu(dx;a,b,c,d)=f(x;a,b,c,d) 1_{|x|\leq1}+\sum_{x\in F(a,b,c,d)}
p(x) \delta_{x}.
\]
Here $F=F(a,b,c,d)$ is a finite or empty set of atoms.
The density $f$ is given by (\ref{AWdistr}). Note that $f$ is
sub-probabilistic for some choices of parameters. The nonobvious fact
that the total mass of $\nu$ is $1$ follows from \cite
{AskeyWilson85}, (2.11), applied to $m=n=0$.

As pointed out by Stokman \cite{Stokman1997}, condition (\ref
{Farvard_condition}) implies that if one of the parameters $a,b,c,d$ has
modulus larger than one, then it must be real. When $m_1=0$, at most two of the four
parameters $a, b, c, d$ have modulus larger than one. If there are two, then
one is positive and the other is negative.

Each of the parameters $a,b,c,d$ that has absolute value larger than
one gives rise to a set of atoms.
For example, if $a\in(-\infty,-1)\cup(1,\infty)$, then the
corresponding atoms are
at
%
%e3.3 ###
%
\begin{equation}
\label{xja}
x_j= \frac{a q^j+(aq^j)^{-1}}{2}
\end{equation}
with $j\geq0$ such that $|q^j a|\geq1$, and the corresponding
probabilities are
%
%e3.5 ###
%e3.4 ###
%
\begin{eqnarray}\label{p_0}\qquad
p(x_0)&=&\frac{(a^{-2},bc,bd,cd)_\infty}{(b/a,c/a,d/a,abcd)_\infty},\\
\label{p_j}
p(x_j)&=&p(x_0) \frac{(a^2,ab,ac,ad)_j(1-a^2
q^{2j})}{(q,qa/b,qa/c,qa/d)_j(1-a^2)} \biggl(\frac{q}{abcd} \biggr)^j,\qquad j\geq0.
\end{eqnarray}
(This formula needs to be re-written in an equivalent form to cover the
cases when $abcd=0$. It is convenient to count as an ``atom'' the case
$|q^j a|= 1$ even though the corresponding probability is $0$. Formula
(\ref{p_j}) incorporates a correction to the typo in \cite
{AskeyWilson85}, (2.10), as in \cite{KoekoekSwarttouw}, Section 3.1.)

The continuous component is completely absent when $K(a,b,c,d)=0$
[recall (\ref{K})]. Although under the assumptions of Theorem \ref
{Thm-QH-q} the univariate distributions are never purely discrete, we
still need to consider the case $K(a,b,c,d)=0$ as we need to allow
transition probabilities of Markov processes to be purely discrete.

%We will need two such cases:
% \item If one of the parameters, say $|a|>1$ is real and the remaining
%satisfy $|b|,|c|,|d|<1$, and $abcd<1$ then $F=\{x_j: j\geq0, |q^j a|>1
%
%$c=q^r a$. If the remaining two parameters are complex conjugate or
%real and satisfy $|b|,|d|<1$, then
%$\nu(dx;a,b,c,d)$ is purely discrete with atoms at \eqref{xja} with
%the probabilities as listed in \eqref{p_0} and \eqref{p_j}.
%pair, say $a,b$ such that $ab\in[1,\infty)$ then
%Suppose that
%the total number of pair-products that violate \eqref{abacad} and
%some $k=0,1,\ldots$. Then $\nu(dx;a,b,c,d)$ exists, and is purely
%atomic with $k+1$ atoms. \fbox{Does
%this hold for $q<0$?}

We remark that if $X$ has distribution $\nu(dx;a,b,c,d)$, then
formulas for ${\mathbb E}(X)$ and $\operatorname{Var}(X)$ from
Proposition \ref{awmom} hold
now for all admissible choices of parameters $a,b,c,d$, as these
expressions can equivalently be derived from
the fact that the first two Askey--Wilson polynomials integrate to zero.

%s3.2 ###
\subsection{Construction of Markov process}
\label{Sec:ConstructionofMarkovprocess}
%As previously, without loss of generality, we assume $|A|\leq|B|$ and
%$|C|\leq|D|$.
%The case $A=B=0$ corresponds to $q$-Meixner processes from
%without loss of generality we assume that $B\ne0$.

Recall $I=I(A,B,C,D;q)$ from (\ref{Imax}). As in Section \ref
{S:elementary}, we first construct the auxiliary Markov process
$(Y_t)_{t\in I}$.
We request that the univariate law $\pi_t$ of $Y_t$ is the
Askey--Wilson law
%
%e3.6 ###
%
\begin{equation}
\label{pi_t}
\pi_t(dy)=\nu\biggl(dy;A\sqrt{t},B\sqrt{t},\frac{C}{\sqrt{t}},\frac
{D}{\sqrt{t}} \biggr).
\end{equation}
In order to ensure that this univariate law exists, we use condition
(\ref{abcd}).
This condition is fulfilled when (\ref{ABCD}) holds and the admissible
range of values of $t$ is the interval
$I$ from (\ref{Imax}).
[The endpoints (\ref{T0}) and (\ref{T1}) were computed by applying M\"
{o}bius transformation (\ref{hhhh}) to the endpoints of $I$.]

For $t\in I$, let $U_t$ be the support of $\pi_t(dy)$. Under the
assumption (\ref{ABCD}), this set can be described quite explicitly
using the already mentioned
results of Askey--Wilson \cite{AskeyWilson85}:
$U_t$ is the union of $[-1,1]$ and
a finite or empty set $F_t$ of points that are of the form
%
%e3.8 ###
%e3.7 ###
%
\begin{eqnarray}
\label{xj}\qquad
x_j(t) &=& \frac12 \biggl(B\sqrt{t}q^j +\frac{1}{B\sqrt{t}q^j} \biggr)
\quad\mbox{or}\quad u_j(t)=\frac12 \biggl(\frac{Dq^j}{\sqrt{t}} +\frac{\sqrt
{t}}{Dq^j} \biggr) \quad\mbox{or}
\\
\label{uj}
y_j(t) &=& \frac12 \biggl(A\sqrt{t}q^j +\frac{1}{A\sqrt{t}q^j} \biggr)
\quad\mbox{or}\quad
v_j(t)=\frac12 \biggl(\frac{Cq^j}{\sqrt{t}} +\frac{\sqrt{t}}{Cq^j} \biggr) .
\end{eqnarray}
There is, at most, a finite number of points of each type. However, not
all such atoms can occur simultaneously. All possibilities are listed
in the following lemma.
\begin{lemma}\label{L-atomy} Under the assumptions of Theorem \ref
{Thm-QH-q}, without loss of generality, assume $|A|\leq|B|$ and
$|C|\leq|D|$. Then the following atoms occur:
\begin{itemize}
\item Atoms $u_j(t)$ appear for $D,C\in\mathbb{R}$, and $t\in I$ that
satisfy $t<D^2$; admissible indexes $j\geq0$ satisfy $D^2q^{2j}>t$.
\item Atoms $v_j(t)$ appear for $D,C\in\mathbb{R}$, and $t\in I$ that
satisfy $t<C^2$; admissible indexes $j\geq0$ satisfy $C^2q^{2j}>t$.
\item Atoms $x_j(t)$ appear for $A,B\in\mathbb{R}$, and $t\in I$ that
satisfy $t>1/B^2$; admissible indexes $j\geq0$ satisfy $tB^2q^{2j}>1$.
\item Atoms $y_j(t)$ appear for $A,B\in\mathbb{R}$, and $t\in I$ that
satisfy $t>1/A^2$; admissible indexes $j\geq0$ satisfy $tA^2q^{2j}>1$.
\end{itemize}
\end{lemma}

(The actual number of cases is much larger as in proofs one needs to
consider all nine possible choices for the end points of the time
interval $I$.)

Next, we specify the transition probabilities of $Y_t$.
\begin{proposition}
\label{Positivity}
For $s<t$, $s,t\in I$ and any real $x\in U_s$ measures
\[
P_{s,t}(x,dy)=\nu\Biggl(dy;A\sqrt{t},B\sqrt{t},\sqrt{\frac
{s}{t}}\bigl(x+\sqrt{x^2-1}\bigr),\sqrt{\frac{s}{t}}\bigl(x-\sqrt{x^2-1}\bigr) \Biggr)
\]
are well defined. Here, if $|x|\leq1$ we interpret $ x\pm\sqrt
{x^2-1}$ as $e^{\pm i\theta_x}=e^{\pm i\arccos(x)}$.
\end{proposition}
\begin{pf}
%Clearly, $P_{s,t}(x,dy)$ for $x\in[-1,1]$ is absolutely continuous and
%has density (\ref{trans-d}).
For $x \in [-1, 1]$, measures $P_{s,t} (x, dy)$ are well defined as
conditions (\ref{abacad}) and (\ref{abcd}) hold.
This covers all possibilities
when $(A,B)$ and $(C,D)$ are conjugate pairs or when
$|A|\sqrt{s},|B|\sqrt{s},|C|/\sqrt{s},|D|/\sqrt{s}<1$, as then
$U_s=[-1,1]$.

It remains to consider $x$ in the atomic part of $\pi_s(dx)$.
Relabeling the parameters if necessary, we may assume $|A|\leq|B|$
and $|C|\leq|D|$.
For each of the cases listed in Lemma \ref{L-atomy},
we need to show that the choice of parameters $a=A\sqrt{t}$,
$b=B\sqrt{t}$, $c=\sqrt{\frac{s}{t}}(x+\sqrt{x^2-1})$,
$d=\sqrt{\frac{s}{t}}(x-\sqrt{x^2-1})$
leads to nonnegative products $\prod_{k=0}^n A_k C_{k+1}\geq0$
[recall (\ref{A_n}) and (\ref{C_n})].
We check this by considering all possible cases for the endpoints of
$I$ and all admissible choices
of $x$ from the atoms of measure $\pi_s$. In the majority of these
cases, condition (\ref{abcd}) holds,
so, in fact,
$A_kC_{k+1}>0$ for all $k$.

Here is one sample case that illustrates what kind of reasoning is
involved in the ``simpler cases'' where (\ref{abcd}) holds and one
example of a more complicated case where (\ref{abcd}) fails.
\begin{itemize}
\item \textit{Case} $CD<0$, $AB<0$, $q\geq0$:
in this case, $A,B,C,D$ are real, $I=(0,\infty)$ and assumption
$ABCD<1$ implies $D^2<1/B^2$. A number of cases arises from Lemma
\ref{L-atomy}, and we present only one of them.
\begin{itemize}
\item \textit{Sub-case} $x=v_j(s)$: then $0<s<C^2$ and the
Askey--Wilson parameters of $P_{s,t}(x;dy)$ are
\[
a=A\sqrt{t},\qquad b=B\sqrt{t},\qquad c=\frac{ q^j C}{\sqrt{t}},\qquad
d=\frac{ {s}}{C q^j\sqrt{t}}.
\]
Thus
\begin{eqnarray*}
ab &=& ABt<0<1,\qquad ac=ACq^j,\qquad ad=\frac{As}{Cq^j},\qquad bc= q^j BC,
\\
bd &=& \frac{Bs}{Cq^j },\qquad cd=s/t<1
\end{eqnarray*}
and
\begin{eqnarray*}
qab &=& qABt<0<1 ,\qquad qac=ACq^{j+1}<ACq^j,\qquad qad=\frac{As}{Cq^{j-1}},\\
qbc &=& q^{j+1} BC,\qquad qbd=\frac{Bs}{Cq^{j-1} },\qquad qcd=qs/t<1 .
\end{eqnarray*}
Since $s<C^2q^{2j}$, this implies $|ad|<|A|\sqrt{s}<1$ as $s<C^2\leq
D^2<1/B^2\leq1/A^2$. For the same reason, $|bd|<1$. Of course,
$|qad|<|ad|<1$, $|qbd|<|bd|<1$.

Finally, since $AC, qAC,BC,qBC<1$, we have $ac,qac,bc,qbc<1$. Thus by
Lemma \ref{L-AW-criter}(i), $P_{s,t}(x,dy)$ is well defined.

\end{itemize}
[We omit other elementary but lengthy sub-cases that lead to
(\ref{abcd}).]

\item \textit{Case} $A,B,C,D\in\mathbb{R}$, $AB>0$, $CD>0$: here
$I(A,B,C,D)=(CD, 1/(AB))$ is nonempty. Again a number of cases
arises from Lemma \ref{L-atomy}, of which we present only one.
\begin{itemize}
\item \textit{Sub-case} $x=x_j(s)$:
here $1/B^2<s<t<1/(AB)$ and $B\sqrt{s}|q^j|\geq1$.
Then the Askey--Wilson parameters of $P_{s,t}(x;dy)$ are
\[
a=A\sqrt{t},\qquad b=B\sqrt{t},\qquad c=\frac{s q^j B}{\sqrt{t}},\qquad
d=\frac{1}{q^j B\sqrt{t}}.
\]
Here, Lemmas \ref{L-AW-criter}(ii) or \ref{L-AW-criter}(iv)
applies with $m_1=2$ and $m_2=0$ when $q\geq0$ or $j$ is even, and
Lemma \ref{L-AW-criter}(v) applies with $m_1=0$, $m_2=2$ when $q<0$
and $j$ is odd.
To see this, we look at the two lists of pairwise products.
\begin{eqnarray*}
ab &=& ABt<1,\qquad ac=ABsq^j<1,\qquad ad=\frac{A}{Bq^j},\qquad bc=s q^j B^2, \\
bd &=& 1/q^j,\qquad cd=s/t<1
\end{eqnarray*}
and
\begin{eqnarray*}
qab &=& qABt<1,\qquad qac=ABsq^{j+1}<1,\qquad qad=\frac{A}{Bq^{j-1}},\\
qbc &=& sq^{j+1} B^2,\qquad qbd=1/q^{j-1},\qquad qcd=qs/t<1 .
\end{eqnarray*}

Since $|A|\leq|B|$ we see that $1/(AB)<1/A^2$ so $|A|<1/\sqrt{s}$ and
$|\frac{A}{Bq^j}|<\frac{\sqrt{s}}{|q^jB|}\leq1$. This shows that
$ad<1$ and $qad<1$.

It is clear that both $bc,bd>1$ when $q>0$ or $j$ is even, and that
$qbc$, $qbd>1$ when $q<0$ and $j$ is odd. Thus by Lemma \ref
{L-AW-criter}, $P_{s,t}(x,dy)$ is well defined.
\end{itemize}
\end{itemize}
%
% \begin{equation}
% \label{Prod>0}
% \prod_{k=0}^n (1-ab q^k)(1-ac q^k)(1-ad q^k)(1-bc q^k)(1-bd q^k)(1-cd
%q^k)\geq0.
% \end{equation}

Other cases are handled similarly and are omitted.
\end{pf}

In the continuous case, $p_t(dy)=p(t,y)\,dy$ and
$P_{s,t}(x,dy)=p(s,x;t,y)\,dy$ correspond to (\ref{marg-d}) and (\ref
{trans-d}), respectively.
We now extend Proposition \ref{mf} to a larger set of measures.
\begin{proposition}
\label{mf-gen} The family of probability measures $\pi_t(dy)$
together with the family of transition probabilities $P_{s,t}(x,dy)$
defines a Markov process $(Y_t)_{t\in I}$ on $\bigcup_{t\in I} U_t$.
That is,
$P_{s,t}(x,dy)$ is defined for all $x\in U_s$. For any $s<t$ from
$I(A,B,C,D)$ and a Borel set $V\subset\mathbb{R}$,
%
%e3.9 ###
%
\begin{equation}\label{marg}
\pi_t(V)=\int_\mathbb{R}P_{s,t}(x,V)\pi_s(dx)
\end{equation}
and for $s<t<u$ from $I$,
%
%e3.10 ###
%
\begin{equation}\label{cond}
P_{s,u}(x,V)=\int_{U_t} P_{t,u}(y,V) P_{s,t}(x,dy) \qquad\mbox{for all
$x\in U_s$} .
\end{equation}
\end{proposition}

To prove this result we follow the same plan that we used in \cite
{BrycMatysiakWesolowski04b} and we will use similar notation. We
deduce all necessary information from the orthogonal polynomials.
Consider two families of polynomials. The first family is
%$$p_n(x;t)=t^{n/2}
%)$$
%
\[
p_n(x;t)=t^{n/2}
\bar w_n(x;a,b,c,d)
\]
with $\bar w_n$ defined by (\ref{AW}) and
%
%e3.11 ###
%
\begin{equation}\label{p-parms}
a=A\sqrt{t},\qquad b=B\sqrt{t},\qquad c=C/\sqrt{t},\qquad d=D/\sqrt{t}.
\end{equation}
The second family is
%
%e3.12 ###
%
\begin{equation}\label{Q_dla_Jacka}
Q_n(y;x,t,s)=t^{n/2} \bar w_n(y;a,b,\widetilde c, \widetilde d ),
\end{equation}
where $a,b$ are in (\ref{p-parms}), and
%
%e3.13 ###
%
\begin{equation}
\label{Q2p_params}
\widetilde c=\sqrt{\frac{s}{t}} \bigl(x+\sqrt{x^2-1} \bigr),\qquad \widetilde
d=\sqrt{\frac{s}{t}} \bigl(x-\sqrt{x^2-1} \bigr).
\end{equation}
%
%$Q_n$ that are
%$$\bar w_n (y;A\sqrt{t},B\sqrt{t},\sqrt{\frac{s}{t}}(x+\sqrt{x^2-1}),
%multiplied by
% \label{Multiplier}
% (\sqrt{\frac{s}{t}}(x+\sqrt{x^2-1}),\sqrt{\frac{s}{t}}(x-

As real multiples of the corresponding Askey--Wilson polynomials
$\widetilde w_n$, polynomials $\{p_n\}$
are orthogonal with respect to $\pi_t(dx)$, and polynomials $\{Q_n\}$
are orthogonal with respect to $P_{s,t}(x,dy)$ when $x\in U_s$.
% the parameters take appropriate values.
%$P_{s,t}(x,dy)$}
It may be interesting to note that if $x$ is in the atomic part of
$U_s$ and $a,b$ from (\ref{p-parms}) satisfy $ab<1<b$, then
%for certain choices of parameters,
the family $\{Q_n\}$ corresponds to a finitely-supported measure. As
explained in Theorem \ref{Thm-Favard} we still have the infinite
family of polynomials $\{Q_n\}$ in this case. This is important to us
as we use this infinite family to infer the Chapman--Kolmogorov
equations and the martingale property of the infinite family $\{p_n\}$.

The following algebraic identity is crucial for our proof.
\begin{lemma} For $n\geq1$,
%
%e3.14 ###
%
\begin{equation}\label{Connection}
Q_n(y;x,t,s)=\sum_{k=1}^n b_{n,k}(x,s) \bigl(p_k(y;t)-p_k(x;s) \bigr),
\end{equation}
where $b_{n,k}(x,s)$ does not depend on $t$ for $1\leq k\leq n$,
$b_{n,n}(x,s)$ does not depend on $x$, and $b_{n,n}(x,s)\ne0$.
\end{lemma}
\begin{pf}When $|A|+|B|\ne0$, due to symmetry, we may assume that
$A\ne0$. From Theorem \ref{Thm-connection}, with
parameters (\ref{p-parms}) and (\ref{Q2p_params}),
we get
%
%e3.15 ###
%
\begin{equation}\label{Qbp}
Q_n(y;x,t,s)=\sum_{k=0}^n b_{n,k} p_k(y;t),
\end{equation}
where\vspace*{1pt} $b_{n,k}=t^{(n-k)/2}\bar c_{k,n} $ is given by (\ref{bar_c_kn}).
Coefficients $b_{n,k}$ do not depend on $t$ as
$t^{(n-k)/2}/a^{n-k}=A^{k-n}$, and $t$ cancels out in all other entries
on the right-hand side of (\ref{connect-AW}):
\begin{eqnarray*}
ab\widetilde c\widetilde d &=& ABs,\qquad abcd=ABCD,\\
a\widetilde c &=& A\sqrt{s}\bigl(x+\sqrt{x^2-1}\bigr),\qquad ac=AC, \\
a\widetilde d &=& A\sqrt{s}\bigl(x-\sqrt{x^2-1}\bigr),\qquad a d=AD.
\end{eqnarray*}
We also see that
\[
b_{n,n}(x,s)=(-1)^n q^{n(n+1)/2}\frac
{(q^{-n},q^{n-1}ABs)_n}{(q,q^{n-1}ABCD)_n}
\]
does not depend on $x$.
Using (\ref{poch2}) we get
\[
b_{n,n}(x,s)=\frac{(q^{n-1}ABs)_n}{(q^{n-1}ABCD)_n},
\]
which is nonzero also when $q=0$.

The case $A=B=0$ is handled similarly, based on (\ref{bar_c_0}). In
this case $b_{n,n}(x,s)=1$.

Next we use (\ref{Q_dla_Jacka}) to show that $Q_n(x;x,s,s)=0$ for
$n\geq1$. We observe that
(\ref{AW}) used for parameters (\ref{p-parms}) and (\ref
{Q2p_params}) gives
$Q_1(x;x,s,s)=0$ as $B_0=a+1/a-A_0/a-C_0a=2x$ when $t=s$, and
$Q_2(x;x,s,s)=0$ as $\bar C_1(a,b,\widetilde c,\widetilde d)=0$ when
$t=s$. So (\ref{AW}) implies that $Q_n(x;x,s,s)=0$ for all $n\geq1$,
and (\ref{Qbp}) implies
\[
\sum_{k=0}^n b_{n,k}(x,s)p_k(x;s)=0.
\]
Subtracting this identity from (\ref{Qbp}) we get (\ref{Connection}).
\end{pf}
%
%%% old proof %%%%%%%%%%%%%%%%%
%An adaptation of \cite[Formula (6.3)]{AskeyWilson85} with $\beta=b$
%shows that
% Q_n(y;x,t,s)=\sum_{k=0}^n b_{n,k}(x,s) p_k(y;t)
%with the connection coefficients
%b_{n,k}(x,s)=(-1)^kq^{k(k+1)/2}(a\gamma,a\delta)_n
%{ _4\varphi_3} (\begin{matrix}
% q^{k-n}, ab\gamma\delta q^{n+k-1},ac q^k,adq^k\\
% q^{2k}abcd, a\gamma q^k, a\delta q^k
% \end{matrix} ;q ),
%which do not depend on $t$, as
%& ab\gamma\delta=ABCD, & abcd=ABs,& \\
% &a\gamma=A\sqrt{s}(x+\sqrt{x^2-1}),& ac=AC,& \\
% &a\delta=A\sqrt{s}(x-\sqrt{x^2-1}),& a d=AD&.
%By \eqref{poch2},
%b_{n,n}(x,s)=(-1)^n q^{n(n+1)/2}
% does not depend on $x$.
%
%We now observe that \eqref{Q-rec} gives
%$Q_1(x;x,s,s)=0$, as $a+1/a-A_0-C_0=2x$ when $t=s$, and
%$Q_2(x;x,s,s)=0$, as $\bar C_1(a,b,c,d)=0$ when $t=s$. So
%$Q_n(x;x,s,s)=0$ for $n\geq1$, and \eqref{Qbp} implies
%$$\sum_{k=0}^n b_{n,k}(x,s)p_k(x;s)=0.$$
%Subtracting this identity from \eqref{Qbp} we get \eqref{Connection}.

We also need the following generalization of the projection formula
\cite{NassrallahRahman85}.
\begin{proposition} \label{L36} Suppose that $A,B,C,D$ satisfy the
assumptions in
Theorem \ref{Thm-QH-q}. For $x\in U_s$,
%
%e3.16 ###
%
\begin{equation}\label{Projection}
\int_\mathbb{R}p_n(y;t)P_{s,t}(x,dy)=p_n(x;s) .
\end{equation}
\end{proposition}
\begin{pf} Since $x\in U_s$, from Proposition \ref{Positivity},
measures $P_{s,t}(x,dy)$ are well defined.
The formula holds true for $n=0$. Suppose it holds true for some
$n\geq0$.
By induction assumption and orthogonality of polynomials $\{Q_n\}$,
\begin{eqnarray*}
0 &=& \int_\mathbb{R}Q_{n+1}(y;x,t,s) P_{s,t}(x,dy)\\
&=& b_{n+1,n+1}(x,s)\int_\mathbb{R}\bigl(p_{n+1}(y;t)-p_{n+1}(x;s) \bigr)
P_{s,t}(x,dy)\\
&=& b_{n+1,n+1}(x,s) \biggl(\int_\mathbb{R}p_{n+1}(y;t) P_{s,t}(x,dy)
-p_{n+1}(x;s) \biggr) .
\end{eqnarray*}
\upqed\end{pf}
\begin{pf*}{Proof of Proposition \protect\ref{mf-gen}}
This proof follows the scheme of the proof of~\cite{BrycMatysiakWesolowski04b},
Proposition 2.5.
To prove (\ref{marg}), let $\mu(V)
=\int_\mathbb{R}P_{s,t}(x,V)\pi_s(dx)$, and note that by orthogonality,
$\int_\mathbb{R}p_n(x;s)\pi_s(dx)=0$ for all $n\geq1$.
Then
from (\ref{Projection}),
\begin{eqnarray*}
\int_\mathbb{R}p_n(y;t)\mu(dy) &=& \int_\mathbb{R}\biggl(\int_{\mathbb{R}} p_n(y;t)
P_{s,t}(x,dy) \biggr)\pi_s(dx)\\
&=& \int_\mathbb{R}p_n(x;s)\pi_s(dx)=0.
\end{eqnarray*}
Since $\int_\mathbb{R}p_n(y;t)\pi_t(dy)=0$, this shows that
all moments
of $\mu(dy)$ and $\pi_t(dy)$ are the same.
By the uniqueness of the moment problem for compactly supported
measures, $\mu(dy)=\pi_t(dy)$,
as claimed.

To prove (\ref{cond}), we first note that for $x\in U_s$,
$P_{s,t}(x,U_t)=1$; this can be seen by analyzing the locations of
atoms, which arise either from the values of $A\sqrt{t}$ or $B\sqrt
{t}>1$ or from $x$ being one of the atoms of
$U_s$. [Alternatively, use (\ref{marg}).]

Fix $x\in U_s$ and let $\mu(V)= \int_{U_t} P_{t,u}(y,V) P_{s,t}(x,dy)$.
Then, by (\ref{Qbp}) for $n\geq1$ and (\ref{Projection}) used twice,
\begin{eqnarray*}
&&\int_\mathbb{R}Q_n(z;x,u,s)\mu(dz) \\
&&\qquad=\int_{U_t}\int_\mathbb
{R}\sum_{k=1}^n
b_{n,k}(x,s) \bigl(p_k(z;u)-p_k(x;s) \bigr) P_{t,u}(y,dz)P_{s,t}(x,dy)\\
&&\qquad=\int_{U_t}\sum_{k=1}^n b_{n,k}(x,s) \bigl(p_k(y;t)-p_k(x;s) \bigr)
P_{s,t}(x,dy)\\
&&\qquad=\int_{\mathbb{R}}\sum_{k=1}^n b_{n,k}(x,s) \bigl(p_k(y;t)-p_k(x;s) \bigr)
P_{s,t}(x,dy)=0 .
\end{eqnarray*}
Thus the moments of $\mu$ and $P_{s,u}(x,dz)$ are equal which, by the
method of moments,
ends the proof.
\end{pf*}

%s3.3 ###
\subsection[Proofs of Theorems 1.1 and 1.2]{Proofs of Theorems \protect\ref{Thm-QH-q} and
\protect\ref{T3}}

\mbox{}

\begin{pf*}{Proof of Theorem \protect\ref{T3}}
If $A, B, C, D$ satisfy the assumptions in
Theorem~\ref{Thm-QH-q}, by Proposition \ref{mf-gen}, there exists a
Markov process $(Y_t)$ with orthogonal
polynomials $\{p_n(x;t)\}$. From (\ref{Projection}) we see that $\{
p_n(x;t)\}$ are
also martingale polynomials for $(Y_t)$. With $Z_t$ defined by
(\ref{Z}), polynomials $r_n(x;t)$ inherit the martingale
property as $r_n(Z_t;t)=p_n(Y_t;t)$.
%
%The proof under the assumptions of
%Proposition \ref{Thm-QH-small} consists of establishing the discrete
%version of the projection formula
% \cite{NassrallahRahman85}. Details are omitted. \fbox{Do we want to
%do this?}
\end{pf*}
\begin{pf*}{Proof of Theorem \protect\ref{Thm-QH-q}}
The fact that time interval $J$ is well defined and nondegenerate has
been shown in Remark \ref{R1.2}. From Proposition \ref{mf-gen}, we
already have the Markov process $(Y_t)$ with Askey--Wilson transition
probabilities. Thus the mean and covariance are (\ref{ordmom1}) and
(\ref{covar}).
Formulas (\ref{Z}) and (\ref{Def:X}) will therefore give us the
process $(X_t)_{t\in J}$ with the correct covariance.

It remains to verify asserted properties of conditional moments. Again
transformation (\ref{Def:X}) will imply (\ref{EQ:q-Var}), provided
$(Z_t)$ satisfies (\ref{Z:q-Var}).
For the proof of the latter we use orthogonal martingale polynomials
(\ref{r_n}). Our proof is closely related to \cite{BrycMatysiakWesolowski04},
Theorem 2.3.
We begin by writing the three step recurrence as
\[
x r_n(x;t)=(\alpha_n t+\beta_n) r_{n+1}(x;t)+(\gamma_n t+\delta_n)
r_{n}(x;t)+(\varepsilon_n t+\varphi_n) r_{n-1}(x;t),
\]
which amounts to decomposing the Jacobi matrix $\mathbf{J}_t$ of $\{
r_n(x;t)\}$ as $t\mathbf{x}+\mathbf{y}$.
%From \eqref{AW}, we read out the coefficients:
From (\ref{AW}) with
$a=A\sqrt{t}$, $b=B\sqrt{t}$, $c=C/\sqrt{t}$ and $d=D/\sqrt{t}$, we read
out the coefficients:
\begin{eqnarray*}
\alpha_n&=&-AB q^n \beta_n ,\\
\beta_n&=&\frac{1-A B C D q^{n-1} }
{\sqrt{1-q} (1-A B C D q^{2 n} ) (1-A B
C D q^{2 n-1} )} ,\\
\varepsilon_n&=&\frac{(1-q^n)
(1-A C
q^{n-1} ) (1-AD
q^{n-1} ) (1-B C
q^{n-1} ) (1-B D q^{n-1} )
}{\sqrt{1-q} (1-A B C D q^{2
n-2} ) (1-A B C D q^{2 n-1} )} ,\\
\varphi_n&=&-CD q^{n-1}\varepsilon_n ,\\
\gamma_n&=& \frac{A}{\sqrt{1-q}}-\frac{\alpha_n}{A} (1-AC q^n)(1-ADq^n)-
\frac{A\varepsilon_n}{(1-AC q^{n-1})(1-ADq^{n-1})}
,\\
\delta_n&=&\frac{1}{A\sqrt{1-q}}-\frac{\beta_n}{A} (1-AC q^n)(1-ADq^n)-
\frac{A\varphi_n}{(1-AC q^{n-1})(1-ADq^{n-1})}
.
\end{eqnarray*}
We note that the expressions for $\gamma_n,\delta_n$ after
simplification\footnote{$
\gamma_n=q^n\frac{ AB (q+1) ((A+B) CD+(C+D) q) q^n-A B C D
(A B (C+D)+(A+B) q) q^{2 n}-(A B (C+D)+(A+B) q)
q}{\sqrt{1-q} (q^2-A B C D q^{2 n} )
(A B C D q^{2 n}-1 )}.
$} are well defined also for $A=0$.
Moreover, by continuity
$\alpha_0,\beta_0,\gamma_0,\delta_0,\varepsilon_0=0,\varphi_0=0$
are defined also at $q=0$.

A calculation verifies the $q$-commutation equation
$[\mathbf{x},\mathbf{y}]_q=\mathbf{I}$ for the two components of the
Jacobi matrix. In terms of the coefficients this amounts to
verification that the expressions above satisfy for $n\geq1$:
%
%e3.21 ###
%e3.20 ###
%e3.19 ###
%e3.18 ###
%e3.17 ###
%
\begin{eqnarray}
\alpha_{n} \beta_{n-1}&=&q \alpha_{n-1} \beta_{n} ,\\
\beta
_{n} \gamma_{n+1}+\alpha_{n} \delta_{n}&=&q (\beta_{n}
\gamma_{n}+\alpha_{n} \delta_{n+1}) ,\\
\gamma_{n} \delta
_{n}+\beta_{n} \varepsilon_{n+1}+\alpha_{n-1} \varphi_{n}&=&q
(\gamma_{n} \delta_{n}+\beta_{n-1} \varepsilon
_{n}+\alpha_{n} \varphi_{n+1})+1 ,\\\delta_{n} \varepsilon
_{n}+\gamma_{n-1} \varphi_{n}&=&q (\delta_{n-1} \varepsilon
_{n}+\gamma_{n} \varphi_{n}) ,\\\varepsilon_{n} \varphi_{n+1}&=&q
\varepsilon_{n+1} \varphi_{n} .
\end{eqnarray}
For a similar calculation see \cite{UchiyamaSasamotoWadati04},
Section 4.2. A more general $q$-commutation equation
$[\mathbf{x},\mathbf{y}]_q=\mathbf{I}+\theta\mathbf{x}+\eta
\mathbf{y}+\tau\mathbf{x}^2+\sigma\mathbf{y}^2$ appears in
\cite{BrycMatysiakWesolowski04}, (2.22)--(2.26).

For compactly supported measures, conditional moments can be now read
out from the properties of the Jacobi matrices; formula (\ref{EQ:LR})
follows from $\mathbf{J}_t=t\mathbf{x}+\mathbf{y}$, and formula
(\ref{EQ:q-Var}) follows from
$[\mathbf{x},\mathbf{y}]_q=\mathbf{I}$.
This can be seen from the proof of \cite{BrycMatysiakWesolowski04},
Lemma 3.4, but for reader's convenience we
include some details.

Denote by $\mathbf{r}(x;t)=[r_0(x;t),r_1(x;t),\ldots]$. Then the
three-step recurrence is
$x \mathbf{r}(x;t)=\mathbf{r}(x;t)\mathbf{J}_t $, and the martingale
polynomial property from Theorem \ref{T3} says that
${\mathbb E}(\mathbf{r}(X_u;u)|\mathcal{F}_t)=\mathbf{r}(X_t;t)$.
(Here we
take all operations componentwise.)

To verify (\ref{EQ:LR}) for compactly supported measures it suffices
to verify that
\[
{\mathbb E}(X_t \mathbf{r}(X_u;u)|\mathcal{F}_s)=\frac
{u-t}{u-s}{\mathbb E}(X_s
\mathbf{r}(X_u;u)|\mathcal{F}_s)+\frac{t-s}{u-s}{\mathbb E}(X_u
\mathbf
{r}(X_u;u)|\mathcal{F}_s).
\]
Using martingale property and the three-step recurrence, this is
equivalent to
\[
\mathbf{r}(X_s;s)\mathbf{J}_t =\mathbf{r}(X_s;s)\biggl(\frac
{u-t}{u-s}\mathbf{J}_s+\frac{t-s}{u-s}\mathbf{J}_u\biggr),
\]
which holds true as
\[
\mathbf{J}_t =\frac{u-t}{u-s}\mathbf{J}_s+\frac{t-s}{u-s}\mathbf{J}_u
\]
for linear expressions in $t$.

To verify (\ref{Z:q-Var}) we write it as
%
%e3.22 ###
%
\begin{eqnarray}
{\mathbb E}(Z_t^2|\mathcal{F}_{s,u})
&=&
\frac{(u-t) (u-q t) Z_s^2}{(u-s) (u-q s)}
+\frac{(q+1) (t-s) (u-t) {Z_u}
{Z_s}}{(u-s) (u-q s)}\nonumber\\[-8pt]\\[-8pt]
&&{} + \frac{(t-s) (t-q s)
Z_u^2}{(u-s) (u-q s)}+\frac{(t-s) (u-t)}{u-q s}.\nonumber
\end{eqnarray}
For compactly supported laws, it suffices therefore to verify that
%
%e3.23 ###
%
\begin{eqnarray}
&&{\mathbb E}(Z_t^2\mathbf{r}(X_u;u)|\mathcal{F}_{s})
\nonumber\\
&&\qquad=
\frac{(u-t) (u-q t)
}{(u-s) (u-q s)}{\mathbb E}(Z_s^2\mathbf{r}(X_u;u)|\mathcal
{F}_{s})\nonumber\\[-8pt]\\[-8pt]
&&\qquad\quad{} + \frac
{(q+1) (t-s) (u-t) }{(u-s) (u-q s)}
{\mathbb E}({Z_u}
{Z_s}\mathbf{r}(X_u;u)|\mathcal{F}_{s})\nonumber\\
&&\qquad\quad{} + \frac{(t-s) (t-q s)
}{(u-s) (u-q s)}+ \frac{(t-s) (u-t)}{u-q s} {\mathbb E}(Z_u^2\mathbf
{r}(X_u;u)|\mathcal{F}_{s}).\nonumber
\end{eqnarray}

Again, we can write this using the Jacobi matrices and martingale
property as
%
%e3.24 ###
%
\begin{eqnarray}\label{QJ}
\mathbf{r}(X_s;s)\mathbf{J}_t^2 &=& \mathbf{r}(X_s;s)
\biggl(\frac{(u-t) (u-qt) }{(u-s)
(u-q s)}\mathbf{J}_s^2\hspace*{-28pt}\nonumber\\
&&\hspace*{41pt}{}
+ \frac{(q+1) (t-s) (u-t) }{(u-s) (u-q s)}
\mathbf{J}_s\mathbf{J}_u+\frac{(t-s) (t-q s)
}{(u-s) (u-q s)}\mathbf{J}_u^2 \biggr)\hspace*{-28pt}\\
&&{}
+ \frac{(t-s) (u-t)}{u-q s}\mathbf{r}(X_s;s).\hspace*{-28pt}\nonumber
\end{eqnarray}

A calculation shows that for $\mathbf{J}_t=t\mathbf{x}+\mathbf{y}$,
the $q$-commutation equation $[\mathbf{x},\mathbf{y}]_q=\mathbf{I}$
is equivalent to
%
%e3.25 ###
%
\begin{eqnarray}\label{J}
\mathbf{J}_t^2
&=& \frac{(u-t) (u-q t) }{(u-s) (u-q
s)}\mathbf{J}_s^2+\frac{(q+1) (t-s) (u-t) }{(u-s) (u-q s)}
\mathbf{J}_s\mathbf{J}_u\nonumber\\[-8pt]\\[-8pt]
&&{} + \frac{(t-s) (t-q s)
}{(u-s) (u-q s)}\mathbf{J}_u^2+\frac{(t-s) (u-t)}{u-q
s}\mathbf{I},\nonumber
\end{eqnarray}
so (\ref{QJ}) holds.

Uniqueness of $(X_t)$ follows from the fact that by
\cite{BrycMatysiakWesolowski04}, Theorem 4.1, each such process has orthogonal
martingale polynomials; from martingale property (\ref{Projection})
all joint moments are determined uniquely and correspond to
finite-dimensional distributions with compactly supported marginals.
%(cf. \cite[Lemma 3.4]{BrycMatysiakWesolowski04})
\end{pf*}

We remark that the following version of Propositions \ref{pfcond}
and \ref{pfcond-d} would shorten the proof of Theorem
\ref{Thm-QH-q}.
\begin{conjecture}
\label{pfcond-gen}
Let $(Y_t)_{t\in I}$ be the Markov process from Proposition \ref
{mf-gen}. Then for
any $s<t<u$ from $I(A,B,C,D)$, the conditional
distribution of $Y_t$ given $\mathcal{F}_{s,u}$
is
%
%e3.26 ###
%
\begin{equation}\label{pfconddis-gen}
\nu\biggl(y; \frac{z\sqrt{t}}{\sqrt{u}} ,\frac{\sqrt{t}}{z\sqrt
{u}},\frac{x\sqrt{s}}{\sqrt{t}} ,
\frac{\sqrt{s}}{x\sqrt{t}} \biggr) .
\end{equation}
(Here, $x=Y_s+\sqrt{Y_s^2-1}$, $ z=Y_u+\sqrt{Y_u^2-1}$.)
\end{conjecture}
%
%From \eqref{pfconddis-gen}, the first two conditional moments are:
%F}_{s,u})=\frac{(u-t)\sqrt{s}Y_s+(t-s)\sqrt{u}Y_u}{\sqrt{t}(u-s)} ,
%F}_{s,u})=\frac{(1-q)(u-t)(t-s)}{t(u-qs)} (\frac{1}{4}
%-\frac{(u\sqrt{s}Y_s-s\sqrt{u}Y_u)(\sqrt{u}Y_u-\sqrt{s}Y_s)}{(u-s)^2}
%) .

%%%%%%%%%%%%%%
%% discrete case begin

%s4 ###
\section{Purely discrete case}\label{S:discrete}
Assumption (\ref{ABCD}) arises from the positivity condition (\ref
{Farvard_condition}) for the
Askey--Wilson recurrence for which it is difficult to give general
explicit conditions. The following result exhibits additional
quadratic harnesses when condition (\ref{ABCD}) is not satisfied.
\begin{proposition}\label{T2}
Suppose $q, A,B,C,D>0$ and $ABCD<1$. Suppose that there are
exactly two numbers among the four products $AC$, $AD$, $BC$, $BD$ that
are larger than one, and that the smaller
of the two, say $AD$, is of the form $1/q^N$ for some integer
$N\geq0$.
If $Aq^N>1$, then there exists a Markov process $(X_t)_{t\in(0,\infty
)}$ with discrete univariate distributions
supported on $N+1$ points such that (\ref{EQ:cov})
(\ref{EQ:LR}) hold,
and (\ref{EQ:q-Var}) holds with parameters $\eta,\theta,\sigma
,\tau,\gamma$ given by (\ref{eta}) through (\ref{gamma}).
\end{proposition}

After re-labeling the parameters, without loss of generality for the
reminder of this section, we will assume that $0<A<B$, $0<C<D$, $AC<1$,
$BC<1$, $AD=1/q^N$, so that $BD>1/q^N$.

%{We remark that discrete Askey-Wilson distributions
%$\nu(dx;A\sqrt{t},B\sqrt{t},C/\sqrt{t},D/\sqrt{t})$ may be well
%defined also
%when $ABCD>1$, a case which we do not consider in this paper.
%}

%s4.1 ###
\subsection{Discrete Askey--Wilson distribution}
The discrete
Askey--Wilson distribution $\nu(dx;a,b,c,d)$
arises in several situations, including the case described in Lemma
\ref{L-AW-criter}(iii). This distribution
was studied in detail by Askey and Wilson \cite{AskeyWilson79}
and was summarized in \cite{AskeyWilson85}.

Here we consider parameters $a,b,c,d>0$ and $0<q<1$ such that
$ad=1/q^N$ and
\begin{eqnarray*}
q^Na&>&1,\qquad q^N/(bc)>1,\qquad ac<1,\\
q^N a/b&>&1,\qquad q^Na/c>1,\qquad q^N ab>1.
\end{eqnarray*}
Note that this implies $abcd<1$ and
\begin{eqnarray*}
ad &=& 1/q^N>1,\qquad ac<1,\qquad bc<1,\qquad bd<1,\qquad cd<1,\\
ab &>& 1/q^N,
\end{eqnarray*}
so from Lemma \ref{L-AW-criter}(iii), the Askey--Wilson law $\nu
(dx;a,b,c,d)=\nu(dx;a,b,c$, $1/(a q^N))$ is well defined and depends on
parameters $a,b,c, q, N$ only and is supported on $N+1$ points,
\[
\{x_k= (q^ka+q^{-k}a^{-1} )/2\dvtx k=0,\ldots,N\} .
\]
%
%According to \cite{AskeyWilson85}, i
According to \cite{AskeyWilson79}, the Askey--Wilson law assigns to
$x_k$ the probability $p_{k,N}(a,b,\break c)=p_k(a,b,c,1/(q^Na))$ [recall
(\ref{p_j})]. The formula simplifies to
%
%e4.1 ###
%
\begin{eqnarray}\label{pmf}\quad
%p_{k,N}(a,b,c)= [\begin{matrix}N
%)_{N-k}}
%{(q^ka^2)_{N+1}} \frac{(ab, ac)_k}{ (\frac{q}{bc} )_N}
% \frac{(1-q^{2k}a^2)q^{\frac{k(k+1)}{2}}}{(-bc)^k}, k=0,\ldots,N .
&&p_{k,N}(a,b,c)\nonumber\\
&&\qquad=\left[\matrix{N\cr k}\right]\frac{ ({q^{k+1}a}/{b},
{q^{k+1}a}/{c} )_{N-k}(ab, ac)_k(1-q^{2k}a^2)q^{{k(k+1)}/{2}}}
{(q^ka^2)_{N+1} ({q}/({bc}))_N (-bc)^k},\\
\eqntext{k=0,\ldots,N .}
\end{eqnarray}
Here $\left[{N}\atop{k}\right]=\frac{(q)_N}{(q)_k(q)_{N-k}}$
denotes the $q$-binomial coefficient.

We remark that if $X$ is a random variable distributed according to
$\nu(dx;a,\break b,c$, $1/(a q^N))$, then ${\mathbb E}(X)$ and $\operatorname
{Var}(X)$ are given by
formulas (\ref{1mom}) and (\ref{awvar}) with $d=1/(q^Na)$,
respectively. This can be seen by a discrete version of the
calculations from the proof of Proposition \ref{awmom};
alternatively, one can use the fact that the first two Askey--Wilson
polynomials, $\bar w_1(X)$ and $\bar w_2(X)$, integrate to zero.

The discrete version of Proposition \ref{techlem} says that with
$d=1/(ma q^N)$,
%
%e4.2 ###
%
\begin{eqnarray}\label{tech-gen}
&&\nu(U;a,b,cm, md)\nonumber\\
&&\qquad=\int\nu\bigl(U; a,b, m\bigl(x+\sqrt{x^2-1}\bigr),m\bigl(x-\sqrt
{x^2-1}\bigr)\bigr)\\
&&\hspace*{42.1pt}{}\times\nu(dx;ma,mb,c,d)\nonumber
\end{eqnarray}
and takes the following form. %\fbox{ $x_k+\sqrt{x_k^2-1}=?$ on the
%support of $\nu(dx;ma,mb,c,d)$}
%
\begin{lemma}
For any $m\in(0,1)$ and any $j=0,1,\ldots,N$,
%
%e4.3 ###
%
\begin{equation}\label{dck}
p_{j,N}(a,b,mc)=\sum_{k=j}^N p_{j,k}(a,b,q^km^2a)p_{k,N}(ma,mb,c) .
\end{equation}
\end{lemma}
\begin{pf}
Expanding the right-hand side of (\ref{dck}) we have
%$$
%)_{k-j}}
%{(q^ja^2)_{k+1}} \frac{(ab, q^km^2a^2)_j}{ (\frac{q}{q^km^2ab} )_k}
% \frac{(1-q^{2j}a^2)q^{\frac{j(j+1)}{2}}}{(-q^km^2ab)^j}
%$$
%$$
% [\begin{matrix}N
%)_{N-k}}
%{(q^km^2a^2)_{N+1}} \frac{(m^2ab, mac)_k}{ (\frac{q}{mbc} )_N}
% \frac{(1-q^{2k}m^2a^2)q^{\frac{k(k+1)}{2}}}{(-mbc)^k}
%$$
%= [\begin{matrix} N
%)_{N-j}(ab, mac)_j(1-a^2q^{2j})q^{\frac{j(j+1)}{2}}}
%{(q^ja^2)_{N+1} (\frac{q}{mbc} )_N(-mcb)^j}
%$$
%)_{N-k}(m^2)_{k-j}
%(mq^jac)_{k-j}}{(q^{k+j}m^2a^2)_{N-j+1} (\frac{q^{j+1}a}{mc} )_{N-j}}
% \frac{(1-q^{2k}m^2a^2)q^{\frac{(k-j)(k-j+1)}{2}}}
%{ (-\frac{mc}{q^ja} )^{k-j}}.
%$$
%WWW re-typeset
%
%e4.4 ###
%
\begin{eqnarray}\label{ow}\hspace*{12pt}
&&\sum_{k=j}^N \left[\matrix{ k
\cr j}\right]\frac{ ({q^{j+1}a}/{b}, {q}/({q^{k-j}m^2}) )_{k-j}}
{(q^ja^2)_{k+1}} \frac{(ab, q^km^2a^2)_j}{ ({q}/({q^km^2ab}) )_k}
\frac{(1-q^{2j}a^2)q^{{j(j+1)}/{2}}}{(-q^km^2ab)^j}
\nonumber\\
&&\quad{}\times\left[\matrix{N
\cr k}\right]\frac{ ({q^{k+1}a}/{b}, {q^{k+1}ma}/{c} )_{N-k}}
{(q^km^2a^2)_{N+1}} \frac{(m^2ab, mac)_k}{ ({q}/({mbc}) )_N}
\frac{(1-q^{2k}m^2a^2)q^{{k(k+1)}/{2}}}{(-mbc)^k}
\nonumber\\
&&\qquad=\left[\matrix{N
\cr j}\right]\frac{ ({q^{j+1}a}/{b}, {q^{j+1}a}/({mc})
)_{N-j}(ab, mac)_j(1-a^2q^{2j})q^{{j(j+1)}/{2}}}
{(q^ja^2)_{N+1} ({q}/({mbc}) )_N(-mbc)^j}
\nonumber\\
&&\qquad\quad{}\times\sum_{k=j}^N\left[\matrix{ N-j
\cr k-j}\right]\biggl( \biggl(\frac{q^{k+1}ma}{c},q^{k+1}a^2 \biggr)_{N-k}(m^2)_{k-j}\\
&&\hspace*{112.4pt}{}\times
(mq^jac)_{k-j}(1-q^{2k}m^2a^2)q^{
{(k-j)(k-j+1)}/{2}}\biggr)\nonumber\\
&&\hspace*{62pt}{}\times\biggl({(q^{k+j}m^2a^2)_{N-j+1} \biggl(\frac{q^{j+1}a}{mc}
\biggr)_{N-j} \biggl(-\frac{mc}{q^ja} \biggr)^{k-j}}\biggr)^{-1}.\nonumber
\end{eqnarray}

Here we used identities (\ref{poch1}) and (\ref{poch2}). The first
one for: (i) $\alpha=q^{j+1}a/b$, $M=k-j$ and $L=N-k$; (ii)
$\alpha=q^km^2a^2$, $M=j$, $L=N-j+1$; (iii) $\alpha=q^ja^2$,
$M=k+1$, $L=N-k$. The second one for: (i) $\alpha=m^2ab$ and
$M=k$; (ii) $\alpha=m^2$ and $M=k-j$.

We transform the sum in (\ref{ow}) introducing $K=k-j$, $L=N-j$,
$\alpha=mq^ja$, $\beta=\frac{m}{q^ja}$ and $\gamma=c$. Then by
(\ref{pmf}) we get
\[
\sum_{K=0}^L \left[\matrix{L
\cr K}\right]\frac{ ({q^{K+1}\alpha}/{\gamma},
{q^{K+1}\alpha}/{\beta} )_{L-K}(\alpha\beta)_{K}
(\alpha\gamma)_K}{(q^K\alpha^2)_{L+1} ({q}/({\beta\gamma}))_L}
\frac{(1-q^{2K}\alpha^2)q^{{K(K+1)}/{2}}}
{(-\beta\gamma)^K}=1 .
\]
Now the result follows since the first part of the expression at the
right-hand side of
(\ref{ow}) is the desired probability mass function.
\end{pf}

%s4.2 ###
\subsection{Markov processes with discrete Askey--Wilson laws}
%Let $A>B,C>0$ and $Aq^N>1$ for some fixed $N>0$.
We now choose the parameters as in Proposition \ref{T2}: $0<A<B$,
$0<C<D=1/(Aq^N)$, $ABCD<1$, $BC<1$,
and choose the time interval $I= (C(q^N A)^{-1},(AB)^{-1} )$
from (\ref{Imax}).
For any $t\in I$, define the discrete distribution
$\pi_t(dx)=\break\sum_{k=0}^N \pi_t(y_k(t)) \delta_{y_k(t)}(dx) $ by
choosing the support from (\ref{uj})
%$$
%y_k(s)= (q^kAs^{\frac{1}{2}}+(q^kA)^{-1}s^{-\frac{1}{2}} )/2 ,
%$$
with weights
%
%e4.5 ###
%
\begin{equation}
\label{*}\pi_t(y_k(t))=p_{k,N} (At^{{1}/{2}},Bt^{
{1}/{2}},Ct^{-{1}/{2}} ) .
\end{equation}

Clearly, the support of $\pi_s$ is $U_s=\{y_0(s),y_1(s),\ldots,y_N(s)\}$.

Also for any $s,t\in I$, $s<t$ and for any
$k\in\{0,1,\ldots,N\}$, define the discrete Askey--Wilson distribution
%$P_{s,t,y_k(s)}=(y_j(t),P_{s,t,y_k(s)}(y_j(t)))_{j=0,1,\ldots,k}$
$P_{s,t}(y_k(s),dy)=\sum_{j=0}^k P_{s,t,y_k(s)}\delta_{y_j(t)}(dy)$
by
%
%e4.6 ###
%
\begin{equation}
\label{**}
P_{s,t,y_k(s)}(y_j(t))=p_{j,k} (At^{{1}/{2}},
Bt^{{1}/{2}},q^kAst^{-{1}/{2}} ) .
\end{equation}

Thus $P_{s,t}(x,dy)$ is defined only for $x$ from the support of $\pi_s$.
Next, we give the discrete version of Proposition \ref{mf}.
\begin{proposition}
The family of distributions $(\pi_t,P_{s,t}(x,dy))$, $s,t\in I$,
$s<t$, $k\in\{0,1,\ldots,N\}$ defines a Markov process
$(Y_t)_{t\in I}$ with trajectories contained in the set of
functions $\{(y_k(t))_{t\in I}, k=0,1,\ldots,N\}$.
\end{proposition}
\begin{pf}
We need to check the Chapman--Kolmogorov conditions. Note that for
any $s<t$ and any $k$ the support of the measure $P_{s,t,y_k(s)}$
is a subset of the support $U_t$ of the measures
$\pi_t$. First we check
%
%e4.7 ###
%
\begin{equation}\label{ck1}
\pi_t(y_j(t))=\sum_{k=j}^N P_{s,t,k}(y_j(t)) \pi_s(y_k(s)),
\end{equation}
which can be written as
\begin{eqnarray*}
&&p_{j,N} (At^{{1}/{2}},Bt^{{1}/{2}},Ct^{-{1}/{2}} )\\
&&\qquad=\sum_{k=j}^N
p_{j,k} (At^{{1}/{2}},Bt^{{1}/{2}},q^kAst^{-{1}/{2}} )
p_{k,N} (As^{{1}/{2}},Bs^{{1}/{2}},Cs^{-{1}/{2}} ) .
\end{eqnarray*}
Now (\ref{ck1}) follows from (\ref{dck}) with
$a=At^{{1}/{2}}$, $b=Bt^{{1}/{2}}$, $c=Cs^{-{1}/{2}}$
and $m=(s/t)^{{1}/{2}}$.

Similarly, the condition
%
%e4.8 ###
%
\begin{equation}\label{ck2}
P_{s,u,u_k(s)}(u_i(u))=\sum_{j=k}^i P_{t,u,u_j(t)}(u_i(u))
P_{s,t,u_k(s)}(u_j(t)),
\end{equation}
assumes the form
\begin{eqnarray*}
&&p_{i,k}(Au^{{1}/{2}},Bu^{{1}/{2}},q^kAsu^{-{1}/{2}})\\
&&\qquad=
\sum_{j=k}^i p_{i,j}(Au^{{1}/{2}},Bu^{
{1}/{2}},q^jAtu^{-{1}/{2}})
p_{j,k}(At^{{1}/{2}},Bt^{{1}/{2}},q^kAst^{-{1}/{2}}) .
\end{eqnarray*}
Therefore (\ref{ck2}) follows from (\ref{dck}) with
$(j,k,N)\to(i,j,k)$, $a=Au^{{1}/{2}}$, $b=Bu^{{1}/{2}}$,
$c=q^kAst^{-{1}/{2}}$ and $m=(t/u)^{{1}/{2}}$ .
\end{pf}

Let $(Y_t)_{t\in I}$ be a Markov process defined by the above
Markov family $(\pi_t,\break P_{s,t,y_k(s)})$.
%(Since $q>0$ and $A>0$, trajectories of $(Y_t)$ are strictly
%increasing.) %if $q>0$
%and $A<0$ then $(Y_t)$ is strictly decreasing. If $q<0$ then
%$(Y_t)$ is neither increasing nor decreasing. Note that if $q=0$
%then $(Y_t)$ is purely deterministic, moving along the path
%$(u_0(t))_{t\in I}$.

Note that at the end-points of $I$, $Y_{C/(q^N A)}$ is degenerate at
$\frac{1+ACq^N}{2(q^N AC)^{{1}/{2}}}$, and $Y_{1/(AB)}$ is
degenerate at $\frac{A+B}{2(AB)^{{1}/{2}}}$ [compare Proposition
\ref{P:bridge}(i)].

Expressions for conditional expectations and conditional variances
are exactly the same as in the absolutely continuous case with $D=q^{-N}A^{-1}$.
\begin{proposition}
For the process $(Y_t)_{t\in I}$ defined above,
\begin{eqnarray*}
{\mathbb E}(Y_t)&=&\frac{(1-q^N)A(Bt+C)-(1-q^N A^2t)(1-BC)}{2At^{
{1}/{2}}(BC-q^N)} ,
\\
\operatorname{Cov}(Y_s,Y_t)
&=&(1-q)(1-q^N)(1-AC)(1-BC)(q^NA-B)(q^N
As-C)\\
&&{}\times(1-ABt)\bigl(4q(st)^{{1}/{2}}(BC-q^N)^2(BC-q^{N-1})\bigr)^{-1},
\\
\operatorname{Var}(Y_t)&=&\frac{(1-AC)(1-BC)(1-q^N)(1-q)(1-ABt)(C-q^N
At)}{4qAt(BC-q^N)^2(BC-q^{N-1})} .
\end{eqnarray*}
\end{proposition}
\begin{pf}
The result follows from the fact that the marginal and conditional
distributions of the process $(Y_t)_{t\in I}$ are finite
Askey--Wilson. Therefore one can apply formulas
(\ref{1mom}) and (\ref{awvar}).
% in the form
% \eqref{expd}
%and \eqref{vard}.
The covariance is derived through conditioning
${\mathbb E}( Y_sY_t) ={\mathbb E}(Y_s{\mathbb E}(Y_t|{\mathcal F}_s))$.
\end{pf}

Since we are interested in the harness properties, we want to find
the conditional distributions of the process with conditioning
with respect to the past and the future, jointly.
The following result says that the conditional distribution
of $Y_t$ given the (admissible) values $Y_s=x,Y_u=z$ is the discrete
Askey--Wilson distribution $\nu(dy;a,b,c,d)$ with parameters
\begin{eqnarray*}
a &=& \sqrt{\frac{t}{u}} \bigl(z+\sqrt{z^2-1}\bigr),\qquad b=\sqrt{\frac{t}{u}}
\bigl(z-\sqrt{z^2-1}\bigr),
\\
c &=& \sqrt{\frac{s}{t}} \bigl(x+\sqrt{x^2-1}\bigr),\qquad d=\sqrt{\frac{s}{t}}
\bigl(x-\sqrt{x^2-1}\bigr)
\end{eqnarray*}
(compare Proposition \ref{pfcond}). Using notation (\ref{pmf}), this
formula takes the following, more concise, form.
\begin{proposition}\label{pfcond-d}
Let $(Y_t)_{t\in I}$ be the Markov process defined by $(\pi_t,\break
P_{s,t,y_k(s)})$ given by (\ref{*}) and (\ref{**}) with
parameters $A,B,C,q,N$. Then for any $s,t,u\in I$
such that $s<t<u$, the conditional distribution of $Y_t$ given
${\mathcal F}_{s,u}$ is defined by the discrete Askey--Wilson
distribution
\[
P\bigl(Y_t=y_j(t)|Y_s=y_k(s),Y_u=y_i(u)\bigr)=p_{j-i,k-i} \biggl(q^iAt^{
{1}/{2}},\frac{t^{{1}/{2}}}{q^iAu}
,\frac{q^kAs}{t^{{1}/{2}}} \biggr).
\]
The expressions for the first two conditional moments are the same
as in the absolutely continuous case with $D=q^{-N}A^{-1}$.
\end{proposition}
\begin{pf}
Due to the Markov property of the process $(Y_t)_{t\in I}$, to
determine the conditional distribution of $Y_t$ given ${\mathcal
F}_{s,u}$, it suffices to find $P(Y_t=y_j(t)|Y_s=y_k(s),Y_u=y_i(u))$
for any $i,j,k\in\{0,1,\ldots,N\}$ such that $i\le j\le k$. Also
the Markov property implies that this probability can be expressed
in terms of conditional probabilities with respect to the past as
\begin{eqnarray*}
p(j|k,i) &=& P\bigl(Y_t=y_j(t)|Y_s=y_k(s),Y_u=y_i(u)\bigr)
\\
&=&
\frac
{P(Y_u=y_i(u)|Y_t=y_j(t))P(Y_t=y_j(t)|Y_s=y_k(s))}{P(Y_u=y_i(u)|Y_s=y_k(s))}
\\
&=&
\frac{p_{i,j}(Au^{{1}/{2}},Bu^{{1}/{2}},q^jAtu^{-{1}/{2}})
p_{j,k}(At^{{1}/{2}},Bt^{{1}/{2}},q^kAst^{-{1}/{2}})}
{p_{i,k}(Au^{{1}/{2}},Bu^{{1}/{2}},q^kAsu^{-{1}/{2}})}.
\end{eqnarray*}

Expanding the expression for the probability mass functions according
to (\ref{pmf}) we
get
\begin{eqnarray*}
p(j|k,i) &=& \left[\matrix{j
\cr i}\right]\frac{ ({q^{i+1}A}/{B}, qu/(t{q^{j-i}} )_{j-i}}
{(q^iA^2u)_{j+1}} \frac{(ABu, q^jA^2t)_i}{ ({q}/({q^jABt}))_j}
\\
&&{}\times
\frac{(1-q^{2i}A^2u)q^{{i(i+1)}/{2}}}{(-q^jABt)^i}
\\
&&{}\times\left[\matrix{k
\cr j}\right]\frac{ ({q^{j+1}A}/{B}, qt/(s{q^{k-j}} )_{k-j}}
{(q^jA^2t)_{k+1}} \frac{(ABt, q^kA^2s)_j}{ ({q}/({q^kABs}))_k}
\\
&&{}\times
\frac{(1-q^{2j}A^2t)q^{{j(j+1)}/{2}}}{(-q^kABs)^j}
\\
&&{}\times\biggl(\left[\matrix{k
\cr i}\right]\frac{ ({q^{i+1}A}/{B}, qu/(s{q^{k-i}} )_{k-i}}
{(q^iA^2u)_{k+1}}\\
&&\hspace*{17.41pt}{}\times \frac{(ABu, q^kA^2s)_i}{ ({q}/({q^kABs}))_k}
\frac{(1-q^{2i}A^2u)q^{{i(i+1)}/{2}}}{(-q^kABs)^i}\biggr)^{-1}.
\end{eqnarray*}

This can be reduced in several steps. The $q$-binomial symbols
reduce, as in the classical ($q=1$) case to
$\left[{k-i}\atop{j-i}\right]$. Then we\vspace*{-1pt} apply (\ref{poch1}) in the
following situations: (i) $\alpha=\frac{q^{i+1}A}{B}$, $M=j-i$,
$L=k-j$; (ii)\vspace*{1pt} $\alpha=q^iA^2u$, $M=j+1$, $L=k-j$; (iii)
$\alpha=q^jA^2t$, $M=i$, $L=k-i+1$; (iv) $\alpha=q^kA^2s$, $M=i$,
$L=j-i$. Also we apply (\ref{poch2}) for (i) $\alpha=\frac{u}{t}$,
$M=j-i$; (ii) $\alpha=ABt$, $M=j$. Thus
\begin{eqnarray*}
p(j|k,i) &=& \left[\matrix{k-i
\cr j-i}\right] \biggl(\frac{qt/s}{q^{k-j}}, q^{i+j+1}A^2u
\biggr)_{k-j}\\[-2pt]
&&{}\times \biggl(\frac{t}{u}
, q^{k+i}A^2s \biggr)_{j-i}(1-q^{2j}A^2t)q^{{(j-i)(j-i+1)}/{2}}\\[-2pt]
&&\hspace*{0pt}{}\times\biggl({(q^{i+j}A^2t)_{k-i+1} \biggl(\frac{q}{q^{k-i}{u}/{s}} \biggr)_{k-i}
\biggl(-q^{k-i}\frac{s}{u} \biggr)^{j-i}}\biggr)^{-1},
\end{eqnarray*}
which, through comparison with the definition (\ref{pmf}), is easily
identified as the distribution we sought.
\end{pf}
\begin{pf*}{Proof of Proposition \protect\ref{T2}}
Since formulas (\ref{1mom}) and (\ref{awvar}) hold for all
Askey--Wilson distributions,
from Proposition \ref{pfcond-d} we see that the conditional moments
and variances in the discrete case are
also given by formulas from Proposition \ref{condmom}.
Therefore the transformed process,
\[
X_t=\frac{2(1+ABt) T(t)^{{1}/{2}}Y_{T(t)}-(A+B)t-(C+{1}/({q^N A}))}
{\sqrt{(1-q)(1-AC)(1-q^{-N})(1-BC)(1-{B}/({q^N A}))}}\sqrt
{1-q^{-N+1}BC} ,
\]
$t\in J$, is a quadratic harness on $J$ with
$\theta,\eta,\tau,\sigma,\gamma$ defined as in the general case
with $D=q^{-N}A^{-N}$. [Recall that $T(t)$ is the M\"{o}bius
transformation (\ref{T(t)}).]
\end{pf*}

%% discrete case end
%%%%%%%%%%%%%

%%%% moved "Examples"

%s5 ###
\section{Some worked out examples}\label{S:WoE}
This section shows how Theorem \ref{Thm-QH-q} is related to some
previous constructions and how it yields new examples. From examples
that have been previously worked out in detail one can see
that the boundary of the range of parameters is not covered by Theorem
\ref{Thm-QH-q}; in particular it does not cover at all the family of
five Meixner L\'{e}vy processes characterized by the quadratic harness
property in \cite{Wesolowski93}. On the other hand, sometimes
new examples arise when processes
run only on a subinterval of $(0,\infty)$.

Theorem \ref{Thm-QH-q} gives $L_2$-continuous processes on an open
interval, so in applications we extend them to the closure of the time domain.

%s5.1 ###
\subsection{$q$-Meixner processes}

Theorem \ref{Thm-QH-q} allows us to extend \cite{BrycWesolowski03},
Theorem~3.5, to negative $\tau$. (The cases $\gamma=\pm1$ which are
included in
\cite{BrycWesolowski03} are not covered by Theorem \ref{Thm-QH-q}.)
\begin{corollary}
Fix $\tau,\theta\in{\mathbb R}$ and $-1<\gamma<1$, and let
\[
T_0=\cases{
0, &\quad if $\tau\geq0$,\cr
-\tau/(1-\gamma), &\quad if $\tau<0$, $\gamma\geq0$,\cr
-\tau, &\quad if $\tau<0$, $\gamma< 0$.}
\]
Then there exists a Markov process $(X_t)$ on $[T_0,\infty)$
such that (\ref{EQ:cov}), (\ref{EQ:LR}) hold, and (\ref{EQ:q-Var})
holds with parameters $\eta=0,\sigma=0$.
\end{corollary}
\begin{pf} Let $q=\gamma$, $A= 0$, $B= 0$, and
\begin{eqnarray*}
C&=&\cases{ \dfrac{-\theta+\sqrt{ \theta
^2-4 \tau}}{2 \sqrt{1-q}}, &\quad $\theta^2\geq4\tau$,\vspace*{2pt}\cr
\dfrac{-\theta+i\sqrt{4 \tau-\theta^2}}{2 \sqrt{1-q}}, &\quad $\theta^2<
4\tau$,}\\
D&=&\cases{ \dfrac{-\theta
-\sqrt{\theta^2-4 \tau}}{2
\sqrt{1-q}}, &\quad $\theta^2\geq4\tau$,\vspace*{2pt}\cr
\dfrac{-\theta-i\sqrt{4\tau-\theta^2}}{2 \sqrt{1-q}}, &\quad
$\theta^2<4\tau$.}
\end{eqnarray*}
Then (\ref{ABCD}) holds trivially, so
by Theorem \ref{Thm-QH-q} and $L_2$-continuity, $(X_t)$ is well
defined on
$\overline J=[T_0,\infty)$. Straightforward calculation of the
parameters from
(\ref{theta}), (\ref{tau}) and (\ref{gamma}) ends the proof.
\end{pf}

When $\tau<0$, the univariate laws of $X_t$ form the
``sixth'' family to be added to the five cases from
\cite{BrycWesolowski03}, Theorem 3.5. The orthogonal polynomials,
with respect to the law of $X_t$, satisfy the recurrence
\[
xp_n(x;t)=p_{n+1}(x;t)+\theta[n]_q
p_n(x;t)+(t+\tau[n-1]_q)[n]_qp_{n-1}(x;t),
\]
where $[n]_q=(1-q^n)/(1-q)$.
So the polynomials with respect to the standardized law of $X_t/\sqrt
{t}$ are
%
%e5.1 ###
%
\begin{eqnarray}\label{free-binomial}
x\widetilde p_n(x;t)&=&\widetilde p_{n+1}(x;t)+\frac{\theta}{\sqrt{t}}
[n]_q \widetilde p_n(x;t)\nonumber\\[-8pt]\\[-8pt]
&&{}+\biggl(1+\frac{\tau}{t} [n-1]_q\biggr)[n]_q \widetilde
p_{n-1}(x;t).\nonumber
\end{eqnarray}

The same law appears under the name $q$-Binomial law in
\cite{SaitohYoshida00b} for parameters $n=-t/\tau\in\mathbb{N}$,
$\tau
=-p(1-p)\in[-1/4,0)$. When $q\leq0$ and $t=|\tau|$, this law is a
discrete law supported on two roots of $\widetilde p_2$ (see Theorem
\ref{Thm-Favard}).

A justification of relating this law to the Binomial can be given for $q=0$.
In this case, recurrence (\ref{free-binomial}) appears in
\cite{BozejkoBryc04}, (3),
with their $a= \frac{\theta}{\sqrt{t}}$ and their $b=\frac{\tau
}{t}$. By \cite{BozejkoBryc04}, Proposition 2.1,
the law $\nu_t$ of $\frac{1}{\sqrt{\tau}}X_t$ is a free convolution
$\frac{t}{|\tau|}$-fold power of the two-point discrete law that
corresponds to $t=-\tau$. That is, $\nu_t=\nu_{-\tau}^{t/|\tau
|\boxplus}$; in particular,
at $t=-n\tau$, $X_t/\sqrt{\tau}$ has the law that is the $n$-fold
free additive convolution of a centered and standardized two-point law.

%s5.2 ###
\subsection{Bi-Poisson processes}

Next we deduce a version of \cite{BrycMatysiakWesolowski04b},
Theorem 1.2.
Here we again have to exclude the boundary cases $\gamma=\pm1$ as
well as
the case $1+\eta\theta= \max\{\gamma,0\}$.
\begin{corollary}
For $-1< \gamma<1$, and $1+\eta\theta> \max\{\gamma,0\}$ there
exists a Markov process $(X_t)_{t\in[0,\infty)}$ such that
(\ref{EQ:cov}), (\ref{EQ:LR}) hold, and (\ref{EQ:q-Var}) holds
with $\sigma=\tau=0$.
\end{corollary}
\begin{pf}
Let $A=0$, $B= -\frac{\eta}{\sqrt{\eta\theta+1-q}}$, $C=0$, $D=
-\frac{\theta
}{\sqrt{\eta\theta+1-q}}$. Then $BD=\frac{\eta\theta}{\eta
\theta+1-q}<1$.
The condition $qBD<1$ is also satisfied as we assume $\eta\theta+1>0$
when $q<0$.
Thus (\ref{ABCD}) holds and we can apply Theorem \ref{Thm-QH-q}.
From formulas (\ref{eta}) through (\ref{gamma}); the quadratic
harness has parameters $\eta,\theta,\sigma=0,\tau=0,\gamma$, as claimed.
\end{pf}

%Other explicit examples that correspond to \cite[Proposition
%4.3]{BrycMatysiakWesolowski04} and \cite[Theorem
%4.5]{BrycMatysiakWesolowski04} can be worked out, but conditions on
%the range of parameters become cumbersome.

%s5.3 ###
\subsection{Free harness}
Next we indicate the range of parameters that guarantee existence of
the processes
described in \cite{BrycMatysiakWesolowski04}, Proposition 4.3.
Let
%
%e5.2 ###
%
\begin{equation}\label{delta1_gamma1}
\alpha=\frac{\eta+\theta\sigma}{1-\sigma\tau},\qquad
\beta=\frac{\eta\tau+\theta}{1-\sigma\tau}.
\end{equation}
\begin{corollary}%\marginpar{\fbox{corrected}}
\label{Thm-FH} For $0\leq\sigma\tau< 1$, $\gamma=-\sigma\tau$, and
$\eta,\theta$ with $2+\eta\theta+2\sigma\tau\geq0$ and
$1+\alpha\beta> 0$, there exists a Markov process
$(X_t)_{t\in[0,\infty)}$ such that (\ref{EQ:cov}), (\ref{EQ:LR})
and (\ref{EQ:q-Var}) hold.
\end{corollary}
\begin{remark} When $2+\eta\theta+2\sigma\tau<0$, two of the
products in (\ref{ABCD}) are in the ``forbidden region'' $[1,\infty)$,
so Theorem \ref{Thm-QH-q} does not apply. However, the univariate
Askey--Wilson distributions are still well defined.
% Corollary \ref{Thm-FH} holds true under sole assumption
% $1+\alpha\beta\geq
% 0$. This more general result cannot be deduced from Theorem
% two of the products in \eqref{ABCD} are in the ``forbidden region"
%$[1,\infty)$.
\end{remark}
\begin{pf*}{Proof of Corollary \protect\ref{Thm-FH}}
Take $q=0$, and let
\begin{eqnarray*}
A &=& -{\frac{\alpha+ \beta\sigma-
{\sqrt{-4\sigma+
{ ( \alpha-
\beta\sigma) }^2}}}
{2{\sqrt{1 + \alpha\beta}}}},\\
B &=& -{\frac{\alpha+ \beta\sigma+
{\sqrt{-4 \sigma+
{ ( \alpha-
\beta\sigma) }^2}}}
{2{\sqrt{1 + \alpha\beta}}}},
\\
C &=& -{\frac{\beta+ \alpha\tau-
{\sqrt{-4\tau+
{ ( \beta-
\alpha\tau) }^2}}}
{2{\sqrt{1 + \alpha\beta}}}},\\
D &=& - {\frac{\beta+ \alpha\tau+
{\sqrt{-4 \tau+
{ ( \beta-
\alpha\tau) }^2}}}
{2{\sqrt{1 + \alpha\beta}}}}.
\end{eqnarray*}

To verify that $AC\notin[1,\infty)$ we proceed as follows. Note
that
%
%e5.6 ###
%e5.5 ###
%e5.4 ###
%e5.3 ###
%
\begin{eqnarray}
\label{A+B}
A+B&=&-\frac{\alpha+\sigma\beta}{\sqrt{1+\alpha\beta}} ,\\
\label{C+D}
C+D&=&-\frac{\alpha\tau+\beta}{\sqrt{1+\alpha\beta}} ,\\
\label{A-B}
A-B&=&\frac{\sqrt{(\alpha-\sigma\beta)^2-4\sigma}}{\sqrt
{1+\alpha
\beta}},\\
\label{C-D}
C-D&=&\frac{\sqrt{(\beta-\tau\alpha)^2-4\tau}}{\sqrt{1+\alpha
\beta}} .
\end{eqnarray}
Multiplying $(A+B)(C+D)$ and $(A-B)(C-D)$ and using
$ABCD=\sigma\tau$, we get
\[
AC
+\frac{\sigma\tau}{AC}-BC-\frac{\sigma\tau}{BC}=\frac{\sqrt
{(\alpha-\sigma\beta)^2-4\sigma}
\sqrt{(\beta-\tau\alpha)^2-4\tau}}{{1+\alpha\beta}}
\]
and
\[
AC
+\frac{\sigma\tau}{AC}+BC+\frac{\sigma\tau}{BC}=\frac{(\alpha
+\sigma\beta)(\alpha\tau+\beta)}{{1+\alpha
\beta}}.
\]
This gives the following quadratic equation for $AC$:
%
%e5.7 ###
%
\begin{equation}
\label{ACfree}\hspace*{28pt}
AC +\frac{\sigma\tau}{AC}=
\frac{(\alpha+\sigma\beta)(\alpha\tau+\beta)+\sqrt{(\alpha
-\sigma\beta)^2-4\sigma}
\sqrt{(\beta-\tau\alpha)^2-4\tau}}{{2(1+\alpha\beta)}}.
\end{equation}
We now note that a quadratic equation $x+a/x=b$ with $0<a<1$ and
complex $b$ can have a root in $[1,\infty)$ only when $b$ is real
and $b\geq1+a$; this follows from the fact that $x+a/x$ is increasing
for $x>a$, so $x+a/x\geq1+a$ for $x\geq1$.
%Proof: x+a/x>1+a$ as

Suppose, therefore, that the right-hand side of (\ref{ACfree}) is
real and larger than $1+\sigma\tau$. Then calculations lead to
$\sqrt{\eta^2-4\sigma}\sqrt{\theta^2-4\tau}\geq
2+\eta\theta+2\sigma\tau$. The right-hand side is nonnegative by
assumption, so squaring the inequality we get $ ( 1 + \alpha
\beta)
{ ( 1 - \sigma\tau) }^2\leq0$
which contradicts the assumption.

Other cases with $AD,BC,BD$ are handled similarly. Since $ABCD=\sigma
\tau<1$ by assumption,
by Theorem \ref{Thm-QH-q} the
quadratic harness exists.

It remains to calculate the parameters. From $AB=\sigma$, $CD=\tau$
we see that (\ref{sigma}) and (\ref{tau}) give the correct values,
and $\gamma=-\sigma\tau$ from (\ref{gamma}).
To compute the remaining parameters, we re-write the expression
under the square root in the denominator of (\ref{eta}) as
\begin{eqnarray*}
&&(1-AC)(1-BC)(1-AD)(1-BD)\\
&&\qquad= \biggl(1+\sigma\tau
-\biggl(AC+\frac{\sigma\tau}{AC}\biggr) \biggr) \biggl(1+\sigma\tau
-\biggl(BC+\frac{\sigma\tau}{BC}\biggr) \biggr).
\end{eqnarray*}
This is the product of two conjugate
expressions [see (\ref{ACfree}), and its derivation]. A calculation
now simplifies the denominator of (\ref{eta}) to
$({1-\sigma\tau})/{\sqrt{1+\alpha\beta}}$.
%$\frac{1-\sigma\tau}{\sqrt{1+\alpha\beta}}$.
Inserting (\ref{A+B})
and (\ref{C+D}), the numerator of (\ref{eta}) simplifies to
${ ( \alpha-
\beta\sigma)
(1 - \sigma\tau) }/\break
{{\sqrt{1 + \alpha\beta}}} $.
The quotient of these two expressions is $\alpha-
\beta\sigma=\eta$. Similar calculation verifies (\ref{theta}).
\end{pf*}
\subsection{Purely quadratic harness}
The quadratic harness with parameters $\eta=\theta=0$ and $\sigma
\tau>0$ has not been previously constructed.
\begin{corollary}\label{Th_PQH}
For $\sigma,\tau>0$ with $\sigma\tau<1$ and $-1<\gamma<1-2\sqrt
{\sigma\tau}$ there
exists a Markov process $(X_t)_{t\in[0,\infty)}$ such that
(\ref{EQ:cov}), (\ref{EQ:LR}) hold, and (\ref{EQ:q-Var}) holds
with $\eta=\theta=0$.
\end{corollary}
\begin{pf}
Let
\[
q=\frac{4(\gamma+ \sigma\tau) }{ (1 + \gamma+
\sqrt{{ ( 1 - \gamma) }^2 - 4\sigma\tau} )^2 }.
\]
To see that $-1<q<1$, note that for $\gamma+\sigma\tau\ne0$,
\[
q=
\frac{ 1 + \gamma^2 - 2 \sigma\tau-
( 1 + \gamma) \sqrt{{ ( 1 - \gamma) }^2 - 4 \sigma\tau} }{2
( \gamma+ \sigma\tau) },
\]
which gives
%
%e5.8 ###
%
\begin{equation}
\label{q-1}
q-1= \frac{ -2\sqrt{ ( 1 - \gamma)^2 - 4 \sigma\tau} }
{ 1 + \gamma+\sqrt{ ( 1 - \gamma) ^2 - 4\sigma\tau} }<0
\end{equation}
and
%
%e5.9 ###
%
\begin{equation}
\label{q+1}
q+1=
{\frac{ 2 ( 1 + \gamma)}{ 1 + \gamma+{\sqrt{{ ( 1 - \gamma) }^2
- 4 \sigma\tau}} }}>0 .
\end{equation}

Noting that $(1-q)^2+4 q\sigma\tau\geq4
\sigma\tau(1-\sigma\tau)>0$, let
\[
A=-B= \frac{i \sqrt{2\sigma}}{\sqrt{{(1 - q )+
\sqrt{{ ( 1 - q ) }^2 + 4 q \sigma\tau}} }}
\]
and
\[
C=-D= \frac{i \sqrt{2\tau}}{\sqrt{{(1 - q )+
\sqrt{{ ( 1 - q ) }^2 + 4 q \sigma\tau}} }}.
\]
Since $A,B,C,D$ are purely imaginary, we only need to verify
condition $BC<1$ which reads
%
%e5.10 ###
%
\begin{equation}
\label{QQQQ} q + 2\sqrt{\sigma\tau} -1< {\sqrt{{ ( 1- q ) }^2 +
4q\sigma\tau}} .
\end{equation}
This is trivially true when $q + 2\sqrt{\sigma\tau} -1<0$. If $q +
2\sqrt{\sigma\tau} -1\geq0$, squaring both sides we get
$
4(1-q)\sqrt{\sigma\tau}>4(1-q)\sigma\tau
$, which holds true as $q<1$ and $0<\sigma\tau<1$.

Thus quadratic
harness
$(X_t)$ exists by Theorem \ref{Thm-QH-q}, and it remains to verify
that its parameters are as claimed. A straightforward calculation shows
that (\ref{sigma}) and (\ref{tau}) give the correct values of
parameters. It remains to verify that formula (\ref{gamma}) indeed
gives the
correct value of parameter $\gamma$.
Since this calculation is lengthy, we indicate major steps: we
write (\ref{gamma}) as (\ref{gamma-1}),
%$\gamma-1=(q-1)(1+ABCD)/(1-qABCD)$,
and evaluate the right-hand side. Substituting values of
$A,B,C,D$ we get
\[
\frac{(q-1)(1+ABCD)}{1-qABCD}=\frac{(1-q)^2+(1+q)\sqrt
{(1-q)^2+4q\sigma\tau}}{2q}.
\]
Then we use formulas (\ref{q-1}) and (\ref{q+1}) to replace $1-q$
and $1+q$ and note that since $\gamma<1-2\sqrt{\sigma\tau}$ we have
$\gamma<1-2\sigma\tau$ and
\[
\sqrt{(1-q)^2+4q\sigma\tau}=\frac{2 (1-\gamma-2 \sigma\tau)}
{\gamma+\sqrt{(1-\gamma) ^2 -4 \sigma\tau
+1}}.
\]
This eventually simplifies the right-hand side of (\ref{gamma-1}) to
$\gamma-1$, so both uses of parameter $\gamma$ are consistent, as claimed.
\end{pf}
%

%%%%% end of examples
%s6 ###
\section{Concluding observations}
This section contains additional observations that may merit further study.

%s6.1 ###
\subsection{Bridge property}
The following proposition lists combinations of parameters that create
a ``quadratic harness bridge'' between either two-point masses, or
degenerated laws.
\begin{proposition}\label{P:bridge}
Let $(Z_t)_{t\in I}$ be the Markov process from Theorem \ref{T3}.
Assume that $AB\ne0$ so that (\ref{Imax}) defines a bounded interval
$I=(S_1,S_2)$
and extend $Z_t$ to the end-points of $I$
by $L_2$-continuity.
\begin{longlist}
\item
If $AB>0$, then $Z_{S_2}=(1/A+1/B)/\sqrt{1-q}$ is deterministic;
similarly, if $CD\geq0$, then
$Z_{S_1}=(C+D)/\sqrt{1-q}$.
\item
If $q\leq0$ and $CD<0$, then $Z_{S_1}$ takes only two-values. Similarly,
if $q\leq0$ and $AB<0$, then $Z_{S_2}$ is a two-valued random variable.
\item If $CD<0$ and $q>0$, then $Z_{0}$ is purely discrete with
the following law:
%
%e6.1 ###
%
\begin{eqnarray}\label{bridgepC}\qquad
\Pr\biggl(Z_0=\frac{q^k C}{\sqrt{1-q}} \biggr) &=& \frac{(AD, BD)_\infty
(AC,BC)_k}{(D/C,ABCD)_\infty(q,qC/D)_k}q^k,\qquad k\geq0,
\\
%
%e6.2 ###
%
\label{bridgepD}
\Pr\biggl(Z_0=\frac{q^k D}{\sqrt{1-q}}\biggr) &=& \frac{(AC, BC)_\infty
(AD,BD)_k}{(C/D,ABCD)_\infty(q,qD/C)_k}q^k ,\qquad k\geq0 .
\end{eqnarray}
%
%(Here, and in the proof of this part we use the notation introduced in
%Section \ref{Sect:AWdens}.)
\end{longlist}
\end{proposition}
\begin{pf}
We can derive the first two statements from moments which are
easier to compute for $(Y_t)$ instead of $(Z_t)$. In the first case,
$\operatorname{Var}(Y_t)=0$ at the endpoints [see (\ref{ordvar})]; in
the second case
$E( \bar w_2^2(Y_t))=0$ at the end-points. Alternatively, one can
compute the limit of the Askey--Wilson law as in the proof of part (iii).

For part (iii), without loss of generality, assume $|A|\leq|B|$ and
$|C|\leq|D|$. Then
the discrete part of $Z_{s}$ has atoms at
\[
\biggl\{\frac{1}{\sqrt{1-q}} \biggl(q^jC +\frac{s}{C q^j} \biggr)\dvtx j\geq0, q^{2j}C^2>
s \biggr\}
\]
and
\[
\biggl\{\frac{1}{\sqrt{1-q}} \biggl(q^jD +\frac{s}{D q^j} \biggr)\dvtx j\geq0, q^{2j}D^2>
s \biggr\}.
\]

The probabilities can be computed from (\ref{p_j}) with
$c=A\sqrt{s},d=B\sqrt{s}$ and either $a=C/\sqrt{s}, b=D/\sqrt{s}$
for (\ref{bridgepC}) or $a=D/\sqrt{s}, b=C/\sqrt{s}$ for
(\ref{bridgepD}) and converge to (\ref{bridgepC}) and
(\ref{bridgepD}), respectively. To see that the limit distribution
is indeed discrete, we note that
\begin{eqnarray*}
&&\sum_{k=0}^\infty\Pr\biggl(Z_0=\frac{q^k C}{\sqrt{1-q}} \biggr)+\Pr\biggl(Z_0=\frac
{q^k C}{\sqrt{1-q}} \biggr)
\\
&&\qquad=\frac{(AD,BD)_\infty}{(D/C,ABCD)_\infty} {_2\varphi_1}
\left(\matrix{
AC, BC\cr
qC/D} ;q\right)\\
&&\qquad\quad{} + \frac{(AC,BC)_\infty}{(C/D,ABCD)_\infty}
{_2\varphi_1}
\left(\matrix{
AD, BD\cr
qD/C} ;q\right)=1.
\end{eqnarray*}
Here we use hypergeometric function notation
%
%e6.3 ###
%
\begin{equation}
\label{hypergeometric}
{_{r+1}\varphi_r}
\left(\matrix{
a_1, a_2, \ldots, a_{r+1}\cr
b_1,b_2,\ldots,b_r}
;z\right)=\sum_{k=0}^\infty\frac{(a_1,a_2,\ldots
,a_r)_k}{(q,b_1,\ldots,b_r)_k}z^k.
\end{equation}
The identity that gives the final equality is
\cite{Ismail05}, (12.2.21), used with $a=AC, b=BC, c=qC/D$.
\end{pf}

%%%%%%%%%%
%From Proposition \ref{pfcond} we verify that Markov processes defined
%in Proposition \ref{mf} may arise
%by conditioning from another such process.
%Suppose $(\widetilde Y_t)_{t\in(0,\infty)}$ is
% the Markov process with
% marginal densities \eqref{marg-d} and transition
%densities \eqref{trans-d} with some parameters $A,B,C,D$ that satisfy
%conditions from Proposition \ref{Thm-QH-small}.
%Fix $0<T_0<T_1$ in $I(A,B,C,D)$ and $y_0,y_1\in(-1,1)$. Then the
%conditional laws $$\pi_t(U):=\Pr(\widetilde Y_t\in U| \widetilde
%Y_{T_0}=y_0,\widetilde Y_{T_1}=y_1)$$ and the conditional laws
%$$P_{s,t}(x,U):=\Pr(\widetilde Y_t\in U| \widetilde Y_{s}=x,\widetilde
%Y_{T_1}=y_1)$$ are
%the univariate laws and the transition probabilities of the Markov
%process $(Y_t)_{t\in(T_0,T_1)}$ corresponding to the parameters $A,B=
%|B|=1/\sqrt{T_1},
%|D|=\sqrt{T_0},
%More generally, we conjecture that all processes $(Y_t)$ with $AB>0$
%and $CD>0$ can be embedded into each other by a similar method, with
%$A,B,C,D$ determined from the equations
%$$AB=1/T_1, A+B=2 y_1/\sqrt{T_1}, CD=T_0, C+D=2 y_0\sqrt{T_0}.$$

%s6.2 ###
\subsection{Transformations that preserve quadratic harness property}
The basic idea behind the transformation (\ref{Def:X}) %in the proof
%of Theorem \ref{Thm-QH-q}
is that if a covariance
%
%e6.4 ###
%
\begin{equation}\label{more-general-cov}
{\mathbb E}(Z_tZ_s)=c_0+c_1\min\{t,s\}+c_2\max\{t,s\}+c_3 ts,
\end{equation}
factors as $(s-\alpha)(1-t\beta)$ for $s<t$ with $\alpha\beta< 1$,
then it can be
transformed into $\min\{t,s\}$ by a deterministic time change and
scaling.

This transformation\vspace*{1pt} is based on the following group action: if
$A=\left[{a\atop c}\enskip{b\atop d}\right]\in GL_2(\mathbb{R})$ is
invertible, then $A$ acts on stochastic processes
$\mathbf{X}=(X_t)$ by $A(\mathbf{X}):=\mathbf{Y}$ with
$Y_t=(ct+d)X_{T_A(t)}$ where $T_A(t)=(at+b)/(ct+d)$ is the
associated M\"{o}bius transformation. It is easy to check that this is
a (right) group action: $A(B(\mathbf{X}))=(B\times A)(\mathbf{X})$.

If ${\mathbb E}(Z_tZ_s)=(s-\alpha)(1-t\beta)$ for $s<t$ and $\alpha
\beta<1$,
then $\mathbf{X}=A^{-1}(\mathbf{Z})$ with $A=\left[{1\atop-\beta
}\enskip{-\alpha\atop1}\right]$ has
${\mathbb E}(X_tX_s)=\min\{s,t\}$. The easiest way to see this is to note
that $T_A$ is increasing for $\alpha\beta<1$, and by group property
$\mathbf{Z}=A(\mathbf{X})$. So with $s<t$,
${\mathbb E}(Z_sZ_t)=(1-s\beta)(1-t\beta){\mathbb
E}(X_{T_A(s)}X_{T_A(t)} )=(1-s\beta
)(1-t\beta)T_A(s)=(s-\alpha)(1-t\beta)$.

It is clear that, at least locally, this group action preserves
properties of linearity of regression and of quadratic conditional
variance. In fact, one can verify that the general form of the
covariance ${\mathbb E}(X_t,X_s)=c_0+c_1\min\{t,s\}+c_2\max\{t,s\}
+c_3 ts$ is
also preserved, and since this covariance corresponds to (\ref{EQ:LR}),
the latter is also preserved by the group action.

%%% end of bbl
%
\begin{appendix}\label{app}
%s7 ###
\section*{Appendix: Supplement on orthogonal polynomials}
%s7.1 ###
\subsection{General theory}
A standard simplifying condition %in most books on
in the general theory of orthogonal polynomials is that the
orthogonality measure has infinite support. This condition may fail for
the transition probabilities of the Markov process in Theorem \ref
{T1}. Since we did not find a suitable reference, for the reader's
convenience we state the general result in the form we need and
indicate how to modify known proofs to cover the case of discrete
orthogonality measure. According to~\cite{AskeyWilson79}, page 1012,
related results are implicit in some of Chebyshev's
work on continued fractions.
\setcounter{theorem}{0}
\begin{theorem}\label{Thm-Favard} Let $A_n,B_n,C_n$ be real, $n\geq0$
and such that
%
%e7.1 ###
\setcounter{equation}{0}
\begin{equation}
\label{Farvard_condition}
\prod_{k=0}^n A_k C_{k+1}\geq0 \qquad\mbox{for all $n\geq0$}.
\end{equation}
Consider two families of polynomials defined by the recurrences
%
%e7.2 ###
%
\begin{eqnarray}\label{FVD-rec}
x\overline{p}_n(x) &=& A_n \overline p_{n+1}(x)+B_n \overline p_n(x)+C_n
\overline p_{n-1}(x),\qquad n\geq0,
\\
%
%e7.3 ###
%
\label{FVD-rec-m}
x {p}_n(x) &=& p_{n+1}(x)+B_n p_n(x)+A_{n-1}C_n p_{n-1}(x),\qquad n\geq0,
\end{eqnarray}
%
%x \hat{p}_n(x)=\sqrt{A_{n}C_{n+1}} \hat p_{n+1}(x)+B_n \hat p_n(x)+
with the initial conditions $p_0=\overline p_0=1$, $p_{-1}=\overline
p_{-1}=0$. Then:
\begin{longlist}
\item %\label{FVD-i}
Polynomials $\{\overline p_n\}$ are well
defined for all $n\geq0$ such that $\prod_{k=0}^{n-1} A_k \ne0$.
(Here and below, the product for $n=0$ is taken as $1$.)
%defined for all $n\geq0$ such that $\prod_{k=-1}^{n-1} A_k C_{k+1}>
%0$.
%
\item
Monic polynomials $\{ p_n\}$ are defined
for all $n\geq0$. For $n$ such that $\prod_{k=0}^{n-1} A_k \ne0$,
the polynomials differ only by normalization
%
%e7.4 ###
%
\begin{equation}\label{bar_p_To_p}
p_n(x)= \overline p_n(x) \prod_{k=0}^{n-1} A_k.
\end{equation}
%
%For $n$ such that solutions $\hat p_n$ of \eqref{FVD-rec-o} are
%defined,
% p_n(x)= \hat p_n(x)\prod_{k=-1}^{n-1} A_k C_{k+1}.

\item
There exists a probability measure $\nu$
such that both families $\{\overline p_n\}$ and $\{ p_n\}$ are
orthogonal with respect to $\nu$. In particular for all $m,n\geq0$,
%
%e7.5 ###
%
\begin{equation}\label{FVD-ortho}
\int p_n(x) p_m(x)\nu(dx)=\delta_{m,n}\prod_{k=0}^{n-1} A_k C_{k+1}.
\end{equation}
%
%Polynomials $ p_n$ and $\overline p_n$ might be zero on the set of $x

Furthermore, if $N$ is the first positive integer such that $A_{N-1} C_{N}=0$,
then
$\nu(dx)$ is a discrete probability measure supported on the finite
set of
$N\geq1$ real and distinct zeros of the polynomial
$ p_{N}$.
\end{longlist}
\end{theorem}
\begin{pf}
It is clear that recurrence (\ref{FVD-rec}) can be solved (uniquely)
for $\overline p_{n+1}$ as long as $A_0,\ldots, A_n\ne0$ while
recurrence (\ref{FVD-rec-m}) has a unique solution for all $n$. It is
also clear that transformation (\ref{bar_p_To_p}) maps the solutions
of recurrence (\ref{FVD-rec}) to the solutions of (\ref{FVD-rec-m}).

If
$\prod_{k=0}^n A_k C_{k+1}> 0$ for all $n$, then each factor
$A_{n-1}C_n$ must be positive, so measure $\nu(dx)$
exists and (\ref{FVD-ortho}) holds for all $m,n$ by Favard's theorem
as stated, for example, in \cite{Ismail05}, Theorem 2.5.2. %
If the product (\ref{Farvard_condition}) is zero starting from
some~$n$, and $N$ is the first positive integer such that $A_{N-1} C_{N}=0$,
then $N\geq1$, and (\ref{FVD-rec-m}) implies that for $n> N$,
polynomial $ p_n$ is divisible by $ p_{N}$.
So if
$\nu(dx)$ is a discrete measure supported on the finite set of %WWW
$N$ zeros of the polynomial
$ p_{N}$, then once we show that the zeros are real, (\ref{FVD-ortho})
holds trivially if either $n\geq N$ or $m\geq N$.
To see that (\ref{FVD-ortho}) holds when $0\leq m,n\leq N-1$, and to
see that all zeros of $p_N$ are distinct and real, we
%modify slightly the proof of \cite[Theorem 1.3.12]{Dunkl2001}.
apply known arguments.
%The proof goes through the following steps.
First, the proof of
% \cite[Theorem 3.2.2]{Szego1939} shows how \eqref{FVD-rec-m}
%determines the Christoffel--Darboux formula as long as long as the
%coefficients $A_{n}C_{n+1}$ are non-zero. In particular,
\cite{Szego1939}, (3.2.4) (or recursion) implies that
%
%e7.6 ###
%
\begin{equation}\label{CDI}
p_n'(x)p_{n-1}(x)- p'_{n-1}(x) p_n(x)>0 \qquad\mbox{for all $x\in\mathbb
{R}$ and all $1\leq n \leq N$,}\hspace*{-32pt}
\end{equation}
so the proof of \cite{Szego1939}, Theorem 3.3.2, establishes
recurrently that each of the polynomials $p_1,\ldots,p_N$ has real and
distinct zeros. Now let $\lambda_0,\ldots,\lambda_{N-1}$ be the
\mbox{zeros} of $p_N$. The remainder of the proof is an adaptation of
the proof of Theorem 1.3.12 in \cite{Dunkl2001}. (Unfortunately, we
cannot apply \cite{Dunkl2001}, Theorem 1.3.12, directly since the $N$th
polynomial is undefined there.) Let $J=[J_{i,j}]$ be the $N\times N$
Jacobi matrix whose nonzero entries are $J_{n,n}=B_n$,
$n=0,1,\ldots,N-1$ and $J_{n,n+1}=1$, $J_{n+1,n}=A_{n}C_{n+1}$,
$n=0,1,\ldots,N-2$. Then (\ref{FVD-rec-m}) says that vector
$\vec{v}_j=[p_0(\lambda _j),\ldots, p_{N-1}(\lambda_j)]^T$ is the
eigenvector of $J$ with eigenvalue $\lambda_j$.

Let $D$ be the diagonal matrix with diagonal entries
\[
d_j= \Biggl(\prod_{k=0}^{j-1}A_k C_{k+1} \Biggr)^{-1/2}>0,\qquad 0\leq j\leq N-1.
\]
Thus $d_0=1$ and $d_{N-1}=(A_0\cdots A_{N-2}C_1\cdots C_{N-1})^{-1/2}$.
Then $D JD^{-1}$ is a symmetric matrix with the eigenvectors $D \vec
{v}_0,\ldots,D \vec{v}_{N-1}$ which correspond to the distinct
eigenvalues $\lambda_0,\ldots,\lambda_{N-1}$.
So the matrix
\[
\biggl[\frac{1}{\|D \vec{v}_0\|}D \vec{v}_0,
\frac{1}{\|D \vec{v}_1\|}D\vec{v}_1,\ldots,\frac{1}{\|D \vec
{v}_{N-1}\| }D \vec{v}_{N-1} \biggr]
\]
%
% $$ [\frac{1}{}D \vec{v}_0{/\|D \vec{v}_0\|},D \vec{v}_1/\|D \vec
%{v}_1\|,\ldots,D \vec{v}_{N-1}/\|D \vec{v}_{N-1}\| ]$$
has orthonormal columns, and hence also orthonormal rows.
The latter gives (\ref{FVD-ortho}) with
$\nu(dx)=\sum_{j=0}^{N-1} \gamma_j\delta_{\lambda_j}$ where
$\gamma_j=(\sum_{k=0}^{N-1} p_k(\lambda_j)^2 d_k^2)^{-1}>0$ (recall
that \mbox{$p_0=1$}). Note that since $d_0=1$,
from (\ref{FVD-ortho}) applied to $m=n=0$ we see that $\sum\gamma
_j=1$, so $\nu$ is a probability measure.
\end{pf}

As an illustration, for the degenerate measure $\mu=\delta_a$, one
has $A_n=0$, $B_n=a$, $C_n=0$. Here, $N=1$, so the family $\{\bar
p_n(x)\}=\{1\}$ consists of just one polynomial, while the monic family
is infinite,
$\{ p_n(x)\dvtx n\geq0\}=\{(x-a)^n\dvtx n\geq0\}$, and $\nu$ is concentrated
on the set of zeros of $p_1=x-a$.

%s7.2 ###
\subsection{Connection coefficients of Askey--Wilson polynomials}
This section contains a re-statement of the special case of
\cite{AskeyWilson85}, formula (6.1), which we need in this paper.
%$$ { _{r+1}\varphi_r} (\begin{matrix}
% a_0, a1,a\ldots,a_r\\
% b_1,\ldots,b_r
% \end{matrix} ;x )=\sum_{k=1}^\infty\frac{ (a_0,a_1,\ldots,a_r )_k}{
%(q,b_1,\ldots,b_r )_k} x^k$$

%The following is a version of \cite[(6.1)]{AskeyWilson85} with
%permuted parameters, and stated in terms of re-scaled version of
%Askey-Wilson polynomials that we use in this paper.
%The connection coefficients are due to Askey and Wilson
%limit.
%
\begin{theorem}\label{Thm-connection}
Let $\{\bar w_n\}$ be defined by (\ref{AW}). If $a\ne0$ then
%
%e7.7 ###
%
\begin{equation}
\label{connect-AW}
\bar w_n(x;a,b,\widetilde c,\widetilde d)=\sum_{k=0}^n \bar c_{k,n}
\bar w_k(x; a,b,c,d),
\end{equation}
where
%
%e7.8 ###
%
\begin{eqnarray}
\label{bar_c_kn}
\bar c_{k,n} &=& (-1)^k q^{k(k+1)/2}
\nonumber\\
&&{}\times\frac{ (q^{-n},q^{n-1}a b \widetilde c \widetilde d
)_k(a\widetilde c,a\widetilde d)_{n}}{a^{n-k}
(q,q^{k-1}abcd,a\widetilde c,a\widetilde d )_k}\\
&&{}\times{ _4\varphi_3}\left(\matrix{
q^{k-n}, ab\widetilde c\widetilde dq^{n+k-1},a c q^k,ad q^k\vspace*{2pt}\cr
abcd q^{2k}, a\widetilde c q^k, a\widetilde d q^k}
;q\right).\nonumber
\end{eqnarray}
[Recall the hypergeometric function (\ref{hypergeometric}).]

If $a=b=0$ and $cd\widetilde d\ne0$, then
%
%e7.9 ###
%
\begin{eqnarray}\label{bar_c_0}
\bar c_{k,n}&=&(-1)^k q^{k(2n+1-k)/2} \frac{(q^{-n})_kd^{n-k}(\widetilde
d/d)_{n-k}}{(q)_k}\nonumber\\[-8pt]\\[-8pt]
&&{}\times{ _2\varphi_1}\left(\matrix{
q^{-n}, \widetilde c/c\vspace*{2pt}\cr
q^{k+1-n}d/\widetilde d}
;q c/\widetilde d\right).\nonumber
\end{eqnarray}
Since $d^m(\widetilde d/d)_m=\prod_{j=0}^{m-1}(d-q^j \widetilde d)$,
expression (\ref{bar_c_0}) is also well defined when $d=0$. Similarly,
it is well defined for $c=0,\widetilde d =0$ [see (\ref{hypergeometric})].
\end{theorem}
\begin{pf}
%With $x=\cos\theta$,
% \label{4fi3}
% \bar w_n(x;a,b,c,d)=\frac{ (ac, ad )_n}{a^n} { _4\varphi_3} (
% q^{-n}, abcdq^{n-1},a e^{i\theta},ae^{-i\theta}\\
% ab, ac, ad
% \end{matrix} ;q ),
The monic form of the Askey--Wilson polynomials $\{\widetilde w_n\}$
and $\{\bar w_n\}$ is the same. Applying (\ref{bar_p_To_p}) twice we
see that
%
%e7.10 ###
%
\begin{equation}\label{AW2AW}
\widetilde w_n(x;a,b,c,d)=(ab)_n\bar w_n(x;a,b,c,d).
\end{equation}
Since our $\widetilde w_n$ is denoted by $p_n$ in \cite
{AskeyWilson85}, formula (\ref{connect-AW}) is recalculated from
\cite{AskeyWilson85}, (6.1), with swapped parameters $a,d$ and with
$\beta=b$, $\gamma=\widetilde c$, $\alpha=\widetilde d$.

To prove the second part, first take $b=0$ and all other parameters
nonzero to write
%
%e7.11 ###
%
\begin{equation}\label{pppp****}\qquad
\bar c_{k,n}=(-1)^k q^{k(k+1)/2}\frac{(a\widetilde c,a\widetilde
d)_n(q^{-n})_k}{a^{n-k}}
{ _3\varphi_2}\left(\matrix{
q^{k-n}, a c q^k,ad q^k\vspace*{2pt}\cr
a\widetilde c q^k, a\widetilde d q^k};q\right).
\end{equation}

Then we apply the limiting case of Sears transformation
\cite{Ismail05}, Theorem 12.4.2,
to rewrite
\[
{_3\varphi_2}\left(\matrix{
q^{k-n}, a c q^k,ad q^k\cr
a\widetilde c q^k, a\widetilde d q^k};q\right)=\frac{(ad q^k)^{n-k}
(\widetilde d/d)_{n-k}}{(a\widetilde d q^k)_{n-k}}
{ _3\varphi_2}\left(\matrix{
q^{k-n}, a d q^k,\widetilde c/c\cr
a\widetilde c q^k, q^{k+1-n}
d/\widetilde d } ;q\right).
\]
This allows us to take the limit $a\to0$ in (\ref{pppp****}), proving
(\ref{bar_c_0}).
\end{pf}
\end{appendix}

\section*{Acknowledgments}
The authors thank %Marek Bo\.{z}ejko for drawing our attention to
%``spurious atom" in the transition
%probabilities of free Brownian motion in \cite[Section 5.3]{Biane98}.
%and
Persi Diaconis, Mourad Ismail, Wojciech Matysiak and Ryszard Szwarc for
helpful discussions and valuable information. We appreciate thorough
comments by the referee which helped to improve the presentation.
%This research was partially supported by NSF
%grant \#DMS-0904720, and by Taft Research Seminar 2008-2009.

% imsref loaded by lrinkeviciute, 2009-11-25 11:22:02
%

%
\printaddresses


\begin{thebibliography}{34}

%b1 ###
\bibitem{Anshelevich01}
%
\begin{barticle}[mr]
\bauthor{\bsnm{Anshelevich},~\bfnm{Michael}\binits{M.}}
(\byear{2003}).
\btitle{Free martingale polynomials}.
\bjournal{J. Funct. Anal.}
\bvolume{201}
\bpages{228--261}.
\bid{doi={10.1016/S0022-1236(03)00061-2}, mr={1986160}}
\end{barticle}
%
\endbibitem

%b2 ###
\bibitem{Anshelevich04IMRN}
%
\begin{barticle}[mr]
\bauthor{\bsnm{Anshelevich},~\bfnm{Michael}\binits{M.}}
(\byear{2004}).
\btitle{Appell polynomials and their relatives}.
\bjournal{Int. Math. Res. Not. IMRN}
\banumber{65}
\bpages{3469--3531}.
\bid{doi={10.1155/S107379280413345X}, mr={2101359}}
\end{barticle}
%
\endbibitem

%b3 ###
\bibitem{AskeyWilson79}
%
\begin{barticle}[mr]
\bauthor{\bsnm{Askey},~\bfnm{Richard}\binits{R.}} \AND
\bauthor{\bsnm{Wilson},~\bfnm{James}\binits{J.}}
(\byear{1979}).
\btitle{A set of orthogonal polynomials that generalize the {R}acah
coefficients or {$6-j$} symbols}.
\bjournal{SIAM J. Math. Anal.}
\bvolume{10}
\bpages{1008--1016}.
\bid{doi={10.1137/0510092}, mr={541097}}
\end{barticle}
%
\endbibitem

%b4 ###
\bibitem{AskeyWilson85}
%
\begin{barticle}[mr]
\bauthor{\bsnm{Askey},~\bfnm{Richard}\binits{R.}} \AND
\bauthor{\bsnm{Wilson},~\bfnm{James}\binits{J.}}
(\byear{1985}).
\btitle{Some basic hypergeometric orthogonal polynomials that generalize
{J}acobi polynomials}.
\bjournal{Mem. Amer. Math. Soc.}
\bvolume{54}
\bpages{iv+55}.
\bid{mr={783216}}
\end{barticle}
%
\endbibitem

%b5 ###
\bibitem{BakryMazet03}
%
\begin{bincollection}[mr]
\bauthor{\bsnm{Bakry},~\bfnm{Dominique}\binits{D.}} \AND
\bauthor{\bsnm{Mazet},~\bfnm{Olivier}\binits{O.}}
(\byear{2003}).
\btitle{Characterization of {M}arkov semigroups on {$\Bbb R$}
associated to
some families of orthogonal polynomials}.
In \bbooktitle{S\'eminaire de {P}robabilit\'es {XXXVII}}.
\bseries{Lecture Notes in Math.}
\bvolume{1832}
\bpages{60--80}.
\bpublisher{Springer}, \baddress{Berlin}.
\bid{mr={2053041}}
\end{bincollection}
%
\endbibitem

%b6 ###
\bibitem{Biane98}
%
\begin{barticle}[mr]
\bauthor{\bsnm{Biane},~\bfnm{Philippe}\binits{P.}}
(\byear{1998}).
\btitle{Processes with free increments}.
\bjournal{Math. Z.}
\bvolume{227}
\bpages{143--174}.
\bid{doi={10.1007/PL00004363}, mr={1605393}}
\end{barticle}
%
\endbibitem

%b7 ###
\bibitem{BozejkoBryc04}
%
\begin{barticle}[mr]
\bauthor{\bsnm{Bo{\.z}ejko},~\bfnm{Marek}\binits{M.}} \AND
\bauthor{\bsnm{Bryc},~\bfnm{W{\l}odzimierz}\binits{W.}}
(\byear{2006}).
\btitle{On a class of free {L}\'evy laws related to a regression problem}.
\bjournal{J. Funct. Anal.}
\bvolume{236}
\bpages{59--77}.
\bid{doi={10.1016/j.jfa.2005.09.010}, mr={2227129}}
\end{barticle}
%
\endbibitem

%b8 ###
\bibitem{BrycMatysiakWesolowski04}
%
\begin{barticle}[mr]
\bauthor{\bsnm{Bryc},~\bfnm{W{\l}odzimierz}\binits{W.}},
\bauthor{\bsnm{Matysiak},~\bfnm{Wojciech}\binits{W.}} \AND
\bauthor{\bsnm{Weso{\l}owski},~\bfnm{Jacek}\binits{J.}}
(\byear{2007}).
\btitle{Quadratic harnesses, {$q$}-commutations, and orthogonal martingale
polynomials}.
\bjournal{Trans. Amer. Math. Soc.}
\bvolume{359}
\bpages{5449--5483}.
\bid{doi={10.1090/S0002-9947-07-04194-3}, mr={2327037}}
\end{barticle}
%
\endbibitem

%b9 ###
\bibitem{BrycMatysiakWesolowski04b}
%
\begin{barticle}[mr]
\bauthor{\bsnm{Bryc},~\bfnm{W{\l}odzimierz}\binits{W.}},
\bauthor{\bsnm{Matysiak},~\bfnm{Wojciech}\binits{W.}} \AND
\bauthor{\bsnm{Weso{\l}owski},~\bfnm{Jacek}\binits{J.}}
(\byear{2008}).
\btitle{The bi-{P}oisson process: A quadratic harness}.
\bjournal{Ann. Probab.}
\bvolume{36}
\bpages{623--646}.
\bid{doi={10.1214/009117907000000268}, mr={2393992}}
\end{barticle}
%
\endbibitem

%b10 ###
\bibitem{BrycWesolowski03}
%
\begin{barticle}[mr]
\bauthor{\bsnm{Bryc},~\bfnm{W{\l}odzimierz}\binits{W.}} \AND
\bauthor{\bsnm{Weso{\l}owski},~\bfnm{Jacek}\binits{J.}}
(\byear{2005}).
\btitle{Conditional moments of {$q$}-{M}eixner processes}.
\bjournal{Probab. Theory Related Fields}
\bvolume{131}
\bpages{415--441}.
\bid{doi={10.1007/s00440-004-0379-2}, mr={2123251}}
\end{barticle}
%
\endbibitem

%b11 ###
\bibitem{BrycWesolowski04}
%
\begin{barticle}[mr]
\bauthor{\bsnm{Bryc},~\bfnm{W{\l}odzimierz}\binits{W.}} \AND
\bauthor{\bsnm{Weso{\l}owski},~\bfnm{Jacek}\binits{J.}}
(\byear{2007}).
\btitle{Bi-{P}oisson process}.
\bjournal{Infin. Dimens. Anal. Quantum Probab. Relat. Top.}
\bvolume{10}
\bpages{277--291}.
\bid{doi={10.1142/S0219025707002737}, mr={2337523}}
\end{barticle}
%
\endbibitem

%b12 ###
\bibitem{Diaconis08r}
%
\begin{barticle}[mr]
\bauthor{\bsnm{Diaconis},~\bfnm{Persi}\binits{P.}},
\bauthor{\bsnm{Khare},~\bfnm{Kshitij}\binits{K.}} \AND
\bauthor{\bsnm{Saloff-Coste},~\bfnm{Laurent}\binits{L.}}
(\byear{2008}).
\btitle{Rejoinder: Gibbs sampling, exponential families and orthogonal polynomials}.
\bjournal{Statist. Sci.}
\bvolume{23}
\bpages{196--200}.
\bid{doi={10.1214/07-STS252}, mr={2446500}}
\end{barticle}
%
\endbibitem

%b13 ###
\bibitem{Dunkl2001}
%
\begin{bbook}[mr]
\bauthor{\bsnm{Dunkl},~\bfnm{Charles~F.}\binits{C.~F.}} \AND
\bauthor{\bsnm{Xu},~\bfnm{Yuan}\binits{Y.}}
(\byear{2001}).
\btitle{Orthogonal Polynomials of Several Variables}.
\bseries{Encyclopedia of Mathematics and Its Applications}
\bvolume{81}.
\bpublisher{Cambridge Univ. Press}, \baddress{Cambridge}.
\bid{mr={1827871}}%
\end{bbook}%
%
\endbibitem%

%b14 ###
\bibitem{Feinsilver86}
%
\begin{barticle}[mr]
\bauthor{\bsnm{Feinsilver},~\bfnm{Philip}\binits{P.}}
(\byear{1986}).
\btitle{Some classes of orthogonal polynomials associated with martingales}.
\bjournal{Proc. Amer. Math. Soc.}
\bvolume{98}
\bpages{298--302}.
\bid{doi={10.2307/2045702}, mr={854037}}
\end{barticle}
%
\endbibitem

%b15 ###
\bibitem{Hammersley}
%
\begin{bincollection}[mr]
\bauthor{\bsnm{Hammersley},~\bfnm{J.~M.}\binits{J.~M.}}
(\byear{1967}).
\btitle{Harnesses}.
In \bbooktitle{Proc. {F}ifth {B}erkeley {S}ymp. {M}ath. {S}tatist.
Probab. {III}: {P}hysical
{S}ciences}
\bpages{89--117}.
\bpublisher{Univ. California Press}, \baddress{Berkeley, CA}.
\bid{mr={0224144}}
\end{bincollection}
%
\endbibitem

%b16 ###
\bibitem{HiaiPetz00}
%
\begin{bbook}[mr]
\bauthor{\bsnm{Hiai},~\bfnm{Fumio}\binits{F.}} \AND
\bauthor{\bsnm{Petz},~\bfnm{D{\'e}nes}\binits{D.}}
(\byear{2000}).
\btitle{The Semicircle Law, Free Random Variables and Entropy}.
\bseries{Mathematical Surveys and Monographs}
\bvolume{77}.
\bpublisher{Amer. Math. Soc.}, \baddress{Providence, RI}.
\bid{mr={1746976}}
\end{bbook}
%
\endbibitem

%b17 ###
\bibitem{Ismail05}
%
\begin{bbook}[mr]
\bauthor{\bsnm{Ismail},~\bfnm{Mourad E.~H.}\binits{M.~E.~H.}}
(\byear{2005}).
\btitle{Classical and Quantum Orthogonal Polynomials in One Variable}.
\bseries{Encyclopedia of Mathematics and Its Applications}
\bvolume{98}.
\bpublisher{Cambridge Univ. Press}, \baddress{Cambridge}.
%Richard A.
%Askey}.
\bid{mr={2191786}}
\end{bbook}
%
\endbibitem

%b18 ###
\bibitem{KoekoekSwarttouw}
%
\begin{bmisc}[auto:SpringerTagBib|2009-01-14|16:51:27]
\bauthor{\bsnm{Koekoek},~\bfnm{R.}\binits{R.}} \AND
\bauthor{\bsnm{Swarttouw},~\bfnm{R.~F.}\binits{R.~F.}}
(\byear{1998}).
\bhowpublished{The Askey scheme of hypergeometric orthogonal
polynomials and its $q$-analogue. Report 98-17, Delft Univ. Technology.
Available at} \url{http://fa.its.tudelft.nl/\textasciitilde koekoek/askey.html}.
\end{bmisc}
%
\endbibitem

%b19 ###
\bibitem{Lytvynov03}
%
\begin{barticle}[mr]
\bauthor{\bsnm{Lytvynov},~\bfnm{Eugene}\binits{E.}}
(\byear{2003}).
\btitle{Polynomials of {M}eixner's type in infinite dimensions---{J}acobi
fields and orthogonality measures}.
\bjournal{J. Funct. Anal.}
\bvolume{200}
\bpages{118--149}.
\bid{doi={10.1016/S0022-1236(02)00081-2}, mr={1974091}}
\end{barticle}
%
\endbibitem

%b20 ###
\bibitem{MansuyYor04}
%
\begin{barticle}[mr]
\bauthor{\bsnm{Mansuy},~\bfnm{Roger}\binits{R.}} \AND
\bauthor{\bsnm{Yor},~\bfnm{Marc}\binits{M.}}
(\byear{2005}).
\btitle{Harnesses, {L}\'evy bridges and {\it{M}onsieur {J}ourdain}}.
\bjournal{Stochastic Process. Appl.}
\bvolume{115}
\bpages{329--338}.
\bid{doi={10.1016/j.spa.2004.09.001}, mr={2111197}}
\end{barticle}
%
\endbibitem

%b21 ###
\bibitem{NassrallahRahman85}
%
\begin{barticle}[mr]
\bauthor{\bsnm{Nassrallah},~\bfnm{B.}\binits{B.}} \AND
\bauthor{\bsnm{Rahman},~\bfnm{Mizan}\binits{M.}}
(\byear{1985}).
\btitle{Projection formulas, a reproducing kernel and a generating
function for
{$q$}-{W}ilson polynomials}.
\bjournal{SIAM J. Math. Anal.}
\bvolume{16}
\bpages{186--197}.
\bid{doi={10.1137/0516014}, mr={772878}}
\end{barticle}
%
\endbibitem

%b22 ###
\bibitem{MasatoshiStokman04}
%
\begin{bincollection}[mr]
\bauthor{\bsnm{Noumi},~\bfnm{Masatoshi}\binits{M.}} \AND
\bauthor{\bsnm{Stokman},~\bfnm{Jasper~V.}\binits{J.~V.}}
(\byear{2004}).
\btitle{Askey--{W}ilson polynomials: An affine {H}ecke algebra approach}.
In \bbooktitle{Laredo {L}ectures on {O}rthogonal {P}olynomials and {S}pecial
{F}unctions}.
\bseries{Adv. Theory Spec. Funct. Orthogonal Polynomials}
\bpages{111--144}.
\bpublisher{Nova Sci. Publ.}, \baddress{Hauppauge, NY}.
\bid{mr={2085854}}
\end{bincollection}
%
\endbibitem

%b23 ###
\bibitem{NualartSchouten00}
%
\begin{barticle}[mr]
\bauthor{\bsnm{Nualart},~\bfnm{David}\binits{D.}} \AND
\bauthor{\bsnm{Schoutens},~\bfnm{Wim}\binits{W.}}
(\byear{2000}).
\btitle{Chaotic and predictable representations for {L}\'evy processes}.
\bjournal{Stochastic Process. Appl.}
\bvolume{90}
\bpages{109--122}.
\bid{doi={10.1016/S0304-4149(00)00035-1}, mr={1787127}}
\end{barticle}
%
\endbibitem

%b24 ###
\bibitem{SaitohYoshida00b}
%
\begin{barticle}[mr]
\bauthor{\bsnm{Saitoh},~\bfnm{Naoko}\binits{N.}} \AND
\bauthor{\bsnm{Yoshida},~\bfnm{Hiroaki}\binits{H.}}
(\byear{2000}).
\btitle{A {$q$}-deformed {P}oisson distribution based on orthogonal
polynomials}.
\bjournal{J. Phys. A}
\bvolume{33}
\bpages{1435--1444}.
\bid{doi={10.1088/0305-4470/33/7/311}, mr={1746884}}
\end{barticle}
%
\endbibitem

%b25 ###
\bibitem{Schoutens00}
%
\begin{bbook}[mr]
\bauthor{\bsnm{Schoutens},~\bfnm{Wim}\binits{W.}}
(\byear{2000}).
\btitle{Stochastic Processes and Orthogonal Polynomials}.
\bseries{Lecture Notes in Statistics}
\bvolume{146}.
\bpublisher{Springer}, \baddress{New York}.
\bid{mr={1761401}}
\end{bbook}
%
\endbibitem

%b26 ###
\bibitem{SchoutensTeugels98}
%
\begin{barticle}[mr]
\bauthor{\bsnm{Schoutens},~\bfnm{Wim}\binits{W.}} \AND
\bauthor{\bsnm{Teugels},~\bfnm{Jozef~L.}\binits{J.~L.}}
(\byear{1998}).
\btitle{L\'evy processes, polynomials and martingales}.
\bjournal{Comm. Statist. Stochastic Models}
\bvolume{14}
\bpages{335--349}.
\bid{mr={1617536}}
\end{barticle}
%
\endbibitem

%b27 ###
\bibitem{SoleUtzet08}
%
\begin{barticle}[mr]
\bauthor{\bsnm{Sol{\'e}},~\bfnm{Josep~Llu{\'{\i}}s}\binits{J.~L.}} \AND
\bauthor{\bsnm{Utzet},~\bfnm{Frederic}\binits{F.}}
(\byear{2008}).
\btitle{On the orthogonal polynomials associated with a {L}\'evy process}.
\bjournal{Ann. Probab.}
\bvolume{36}
\bpages{765--795}.
\bid{doi={10.1214/07-AOP343}, mr={2393997}}
\end{barticle}
%
\endbibitem

%b28 ###
\bibitem{SoleUtzet08b}
%
\begin{barticle}[mr]
\bauthor{\bsnm{Sol{\'e}},~\bfnm{Josep~Llu{\'{\i}}s}\binits{J.~L.}} \AND
\bauthor{\bsnm{Utzet},~\bfnm{Frederic}\binits{F.}}
(\byear{2008}).
\btitle{Time--space harmonic polynomials relative to a {L}\'evy process}.
\bjournal{Bernoulli}
\bvolume{14}
\bpages{1--13}.
\bid{doi={10.3150/07-BEJ6173}, mr={2401651}}
\end{barticle}
%
\endbibitem

%b29 ###
\bibitem{Stokman1997}
%
\begin{barticle}[mr]
\bauthor{\bsnm{Stokman},~\bfnm{Jasper~V.}\binits{J.~V.}}
(\byear{1997}).
\btitle{Multivariable {$BC$} type {A}skey-{W}ilson polynomials with partly
discrete orthogonality measure}.
\bjournal{Ramanujan J.}
\bvolume{1}
\bpages{275--297}.
\bid{doi={10.1023/A:1009757113226}, mr={1606922}}
\end{barticle}
%
\endbibitem

%b30 ###
\bibitem{StokmanKoornwinder98}
%
\begin{barticle}[mr]
\bauthor{\bsnm{Stokman},~\bfnm{J.~V.}\binits{J.~V.}} \AND
\bauthor{\bsnm{Koornwinder},~\bfnm{T.~H.}\binits{T.~H.}}
(\byear{1998}).
\btitle{On some limit cases of {A}skey--{W}ilson polynomials}.
\bjournal{J. Approx. Theory}
\bvolume{95}
\bpages{310--330}.
\bid{doi={10.1006/jath.1998.3209}, mr={1652876}}
\end{barticle}
%
\endbibitem

%b31 ###
\bibitem{Szego1939}
%
\begin{bbook}[auto:SpringerTagBib|2009-01-14|16:51:27]
\bauthor{\bsnm{Szeg{\"{o}}},~\bfnm{G.}\binits{G.}}
(\byear{1939}).
\btitle{Orthogonal {P}olynomials}.
\bseries{American Mathematical Society Colloquium Publications}
\bvolume{23}.
\bpublisher{Amer. Math. Soc.}, \baddress{New York}.
\end{bbook}
%
\endbibitem

%b32 ###
\bibitem{UchiyamaSasamotoWadati04}
%
\begin{barticle}[mr]
\bauthor{\bsnm{Uchiyama},~\bfnm{Masaru}\binits{M.}},
\bauthor{\bsnm{Sasamoto},~\bfnm{Tomohiro}\binits{T.}} \AND
\bauthor{\bsnm{Wadati},~\bfnm{Miki}\binits{M.}}
(\byear{2004}).
\btitle{Asymmetric simple exclusion process with open boundaries and
{A}skey--{W}ilson polynomials}.
\bjournal{J. Phys. A}
\bvolume{37}
\bpages{4985--5002}.
\bid{doi={10.1088/0305-4470/37/18/006}, mr={2065218}}
\end{barticle}
%
\endbibitem

%b33 ###
\bibitem{Wesolowski93}
%
\begin{barticle}[mr]
\bauthor{\bsnm{Weso{\l}owski},~\bfnm{Jacek}\binits{J.}}
(\byear{1993}).
\btitle{Stochastic processes with linear conditional expectation and quadratic
conditional variance}.
\bjournal{Probab. Math. Statist.}
\bvolume{14}
\bpages{33--44}.
\bid{mr={1267516}}
\end{barticle}
%
\endbibitem

%b34 ###
\bibitem{Williams73}
%
\begin{bincollection}[mr]
\bauthor{\bsnm{Williams},~\bfnm{David}\binits{D.}}
(\byear{1973}).
\btitle{Some basic theorems on harnesses}.
In \bbooktitle{Stochastic Analysis (a Tribute to the Memory of {R}ollo
{D}avidson)}
\bpages{349--363}.
\bpublisher{Wiley}, \baddress{London}.
\bid{mr={0362565}}
\end{bincollection}
%
\endbibitem

\end{thebibliography}
\end{document}